\documentclass[amssymb,reqno,amsfonts,refcheck,12pt,verbatim,righttag]{amsart}

\usepackage{graphicx}
\usepackage{graphics,color}
\usepackage{amssymb,mathrsfs}
\usepackage{graphicx,color,tikz,caption,subcaption}
\usepackage{mathtools}
\usepackage{enumerate}
\usepackage{amsmath}
\usepackage{bbm}

\makeatletter
\@namedef{subjclassname@1991}{2020 Mathematics Subject Classification}
\makeatother

\usepackage[colorlinks, linkcolor=blue,citecolor=blue, anchorcolor=blue,urlcolor=blue,hypertexnames=false]{hyperref}

\numberwithin{equation}{section} 
\newtheorem{thm}{Theorem}[section]
\newtheorem{lem}[thm]{Lemma}
\newtheorem{pro}[thm]{Proposition}

\newtheorem{de}[thm]{Definition}
\newtheorem{rem}[thm]{Remark}

\def\Z{\mathbb Z}
\def\R{\mathbb R}
\def\N{\mathbb N}

\def\H{\mathcal H}
\def\A{\mathcal A}

\newcommand\hdim{\dim_{\mathcal H}}
\def\F{\mathcal F}
\def\G{\mathcal G}
\def\K{\mathcal K}

\def\x{\mathbf x}
\def\p{\mathbf p}

\usepackage{float}

\begin{document}
\baselineskip 14pt
\title{Hausdorff measures of sets in Exact Diophantine approximation}
\author{Bo Tan, Chen Tian$^{\dag}$,  Baowei Wang and Jun Wu}

\address{Tan, Wang, Wu: School  of  Mathematics  and  Statistics,
                Huazhong  University  of Science  and  Technology, 430074 Wuhan, PR China}
          \email{tanbo@hust.edu.cn, bwei\_wang@hust.edu.cn, jun.wu@hust.edu.cn}
\address{Tian: School of Statistics and Mathematics, Hubei University of Economics, 430205 Wuhan, PR China $^{1}$\\
  Hubei Center for Data and Analysis, 430205 Wuhan, PR China $^{2}$}
\email{tchen@hbue.edu.cn}


\thanks{$^{\dag}$Corresponding author.}

\date{}

\newcommand{\RR}{\mathbb{R}}

\begin{abstract}
Let $(X, d)$ be a compact metric space, and let $Q \subset X$ be countable. Given functions $R: Q \to \RR^+$ and $\phi: \RR^+ \to \RR^+$, we consider the set $E(Q, R, \phi)$ of points $x \in X$ that ``hit'' the shrinking balls $B({\xi},{\phi(R(\xi))})$ for infinitely many $\xi \in Q$, yet,  for every $\epsilon \in (0,1)$, are eventually ``cleared out'' from the slightly smaller neighborhoods $B({\xi},{(1-\epsilon)\phi(R(\xi))})$, that is, they lie outside all but finitely many of these smaller balls.

We give sufficient conditions (also necessary under mild assumptions) for $E(Q, R, \phi)$ to have infinite Hausdorff $f$-measure. This setting generalizes both the classical set $\mathrm{Exact}(\psi)$ of exactly $\psi$-approximable points (with $\psi$ non-increasing) and certain types of restricted Diophantine approximation sets.

\end{abstract}

\maketitle

\section{Introduction}
\subsection{Background}\label{back}
In this article, we compute the Hausdorff measure of ``exactly approximable'' sets within a general framework. To set the scene for the abstract setup considered herein, we begin by recalling the classical set of $\psi$-well approximable numbers. Given a non-increasing function $\psi:\R^+\rightarrow \R^+,$  we define the set of $\psi$-well approximable numbers as 
$$
W(\psi) := \left\{x \in [0,1]: |x-p/q| < \psi(q) \text{ for i.m. } (p,q)\in\Z\times \N \right\},$$
here and throughout, ``i.m.'' abbreviates ``infinitely many''.
The sets $W(\psi)$ are the cornerstone of the Diophantine approximation, and their properties have been comprehensively studied.  Khintchine \cite{K24} showed that $W(\psi)$ is a Lebesgue  null set if and only if
$\sum_{n=1}^{\infty}n\psi(n) < \infty.$  Jarn\'{\i}k \cite{J29} and independently Besicovitch  \cite{B34} proved that for $\tau\ge 2,$
$$\hdim W(x \mapsto x^{-\tau})=\frac{2}{\tau}, $$
where  $\hdim$ denotes the Hausdorff dimension. Furthermore, Jarn\'{\i}k \cite{J31} advanced the theory by proving that  the Hausdorff  $f$-measure  $\H^f( W(\psi))$ obeys a `zero-infinity' law.

For a point $ x \in W(\psi),$ a natural question is whether $ \psi $ is, in some sense, its optimal approximation function. To formalize  this, we study the set of points with $\psi$-exact approximation, defined by
$$\text{Exact}(\psi):=W(\psi)\setminus \bigcup_{0<\epsilon<1}W((1-\epsilon) \psi).$$ For any $x \in \text{Exact}(\psi)$, the function $\psi$ is   considered its optimal approximation function. This follows from the observation  that $x$ lies in $W(\psi)$ but in none of the smaller sets $W((1-\epsilon) \psi)$ for $0 < \epsilon < 1.$

Under the assumption that $\psi$ is non-increasing and $\psi(x)=o(x^{-2})$,
Jarn\'{\i}k \cite{J31} first established the no-emptiness of
$\operatorname{Exact}(\psi)$.
Bugeaud \cite{B03} later significantly strengthened this by determining
the Hausdorff dimension: if $x \mapsto x^2\psi(x)$ is non-increasing and
$\sum_{q=1}^\infty q\psi(q) < \infty$, then
\begin{align*}
\hdim \text{Exact}(\psi) = \hdim W(\psi) = \frac{2}{\lambda},
\end{align*}
where $\lambda = \liminf_{x\rightarrow \infty} {-\log \psi(x)}/{\log x}$. This dimensioanlity result was later relaxed to Jarn\'{i}k's original assumptions by Bugeaud and Moreira \cite{BM11}. We note that the special case 
was resolved earlier by G\"{u}ting \cite{G63}, who showed  for $\tau\ge2$ that   $\hdim \text{Exact}(x\mapsto x^{-\tau}) = {2}/{\tau}$, a result which Beresnevich, Dickinson, and Velani \cite{BDV01} later refined using a  logarithmic scale. 
For further details on this topic, we refer readers to \cite{BGN23,BDV01,B08,BM11,M18}, while the higher-dimensional version of this problem is discussed  in \cite{BS23,F24}.
The Hausdorff dimension of $\mathrm{Exact}(\psi)$-type sets has also been studied in various frameworks, including the complex numbers \cite{HX22}, $\beta$-expansions \cite{ZZ24}, continued fractions \cite{TTZ25}, and the field of formal series \cite{P25, Z12}. Other fractal characteristics of $\mathrm{Exact}(\psi)$, such as its packing dimension \cite{SJ23} and Fourier dimension \cite{FW23, FW24}, have likewise been the subject of investigation.



The set of $\psi$-badly approximable points,  introduced by Beresnevich, Dickinson and Velani \cite{BDV01}, is  defined as
$$
\text{Bad}(\psi):=W(\psi) \setminus \bigcap_{0<\epsilon<1}W((1-\epsilon)\psi).
$$
It is closely related to $\text{Exact}(\psi)$ via the trivial inclusions
$$\text{Exact}(\psi) \subset \text{Bad}(\psi) \subset W(\psi),$$
which immediately implies that $\hdim\text{Exact}(\psi)\le \hdim\text{Bad}(\psi).$ A line of research has focused on   determining $\hdim\text{Bad}(\psi)$ for general   $\psi$.  The equality $\hdim \mathrm{Bad}(\psi)=\hdim W(\psi)$ was first established for power functions $\psi(q) = q^{-\tau}$ with $\tau \in (2, +\infty]$  by Koivusalo, Levesley, Ward and Zhang \cite{KLWZ24}. Bandi and de Saxc\'{e} \cite{BS23} proved the same result for non-increasing functions $\psi$ satisfying $\psi(q) = o(q^{-2})$ using different techniques.  Schleischitz \cite{SJ25} later extended these findings to arbitrary decreasing functions $\psi$, expressing the  dimension   in terms of the lower order of $1/\psi$ at infinity. Wang and Wu \cite{WW24} determined the Hausdorff measure of the $\psi$-badly approximable set within the framework introduced by  Bandi, Ghosh and Nandi \cite{BGN23}. For comprehensive historical background, we refer to  \cite{BS23, BV08,B03, BM11,KLWZ24,SJ25}.


\subsection{A general framework for exactness} We consider a metric space  $(X,d)$,    a countable set $Q \subset X$, and    a  function $R:Q \rightarrow (0,1)$. As a standing assumption in our framework (to avoid pathological cases), we require that for every  $M>0$, the set $\{\xi \in Q: R(\xi)\ge M\}$ is finite. For  a function $\phi : \mathbb{R}^+ \to \mathbb{R}^+$, we define the $\phi$-well approximable set (with respect to $(Q,R)$) to be
 $$
W(Q,R,\phi) := \left\{x \in X : d(x,\xi) < \phi(R(\xi)) \text{ for infinitely many } \xi \in Q\right\}.$$
The corresponding exact set is then defined by
 $$E(Q,R,\phi):=W(Q,R,\phi)\setminus \bigcup_{0<\epsilon<1} W(Q,R,(1-\epsilon) \phi). $$

 The framework we adopt is similar to that introduced by Bandi, Ghosh, and Nandi \cite{BGN23}. This general setup encompasses, among others,  the classic sets of exactly approximable numbers and various   restricted Diophantine approximation sets; see \S \ref{appl} for details.

As a first illustration of our  framework, we recover the classical set $\mathrm{Exact}(\psi)$ from \S \ref{back} by taking $(X,d)=([0,1],|\cdot|)$, $Q = \{ {p}/{q} : (p,q) \in \mathbb{Z} \times \mathbb{N} \}$, $R(p/q) = 1/q^2$, and a non-decreasing function $\phi: \mathbb{R}^+ \to \mathbb{R}^+$ such that $ \phi(1/q^2)=\psi(q)$.
%

For this foundational example, the first two authors together with Zhou \cite{TTZ25+} computed the $s$-dimensional Hausdorff measure of $\mathrm{Exact}(\psi)$ using the theory of continued fractions. Nevertheless, this approach fails in more general settings---including Euclidean spaces of dimension $n \geq 2$---due to the absence of a comparable analogue to continued fractions. The principal challenge in determining the Hausdorff $f$-measure of the general set $E(Q,R,\phi)$ lies in its complex structure: it is not a pure $\limsup$ set, but rather the intersection of a $\limsup$ set and a $\liminf$ set. This complexity stems from the requirement that, for all but finitely many $\xi \in Q$, points satisfying $d(x,\xi) < (1 - \epsilon)\phi(R(\xi))$ must be excluded.

In the classical setting, the works of \cite{B03} and \cite{TTZ25+} show that $W(\psi)$ and $\mathrm{Exact}(\psi)$ typically share $s$-dimensional Hausdorff measures, and thus the Hausdorff dimension. This phenomenon can be attributed to the fact that each $W(\psi)$ is sufficiently well distributed in the unit interval---exhibiting, for instance, full packing dimension and the large intersection property \cite{F94}. This observation motivates a fundamental conjecture: under a local ubiquity hypothesis, and assuming $Q$ is well-separated with mild regularity conditions on $(X, d, \mu)$, the sets $W(Q, R, \phi)$ and $E(Q, R, \phi)$ should share the same Hausdorff dimension and Hausdorff $f$-measure.  Our main result, Theorem \ref{mainthm}, offers compelling evidence in favor of this conjecture.

  We now state our main result,  but postpone the necessary technical definitions (Definitions \ref{as21} and \ref{as22}) to the next section.


%

\begin{thm}\label{mainthm}
Let $(X,d, \mu)$ be a $\delta$-regular metric measure space. Let $u=\{u_n\}$ be a   decreasing sequence which converges to 0 and satisfies that    $\limsup_{n \to \infty} u_n/u_{n+1} < \infty$. Let $\rho$ be a $u$-regular function. Let $(Q,R)$ be a well-distributed system relative to $(\rho, u)$, and let $\phi:\mathbb{R}^{+}\rightarrow \mathbb{R}^{+}$ be a non-decreasing function such that
$$\lim_{r\rightarrow 0} \frac{\phi(r)}{r} = 0,~~~ \sum_{\xi\in Q} \left( \phi(R(\xi)) \right)^{\delta} < \infty, \text{ and } \limsup_{n \rightarrow \infty} \frac{\phi(u_n)}{\rho(u_n) }= 0.$$ Suppose that $f$ is a dimension function with $\lim_{r\rightarrow 0}r^{-\delta}f(r)=\infty.$   Define
 $  g(r):={f(\phi(r))}/{\rho(r)^\delta}. $
If $\sum_{n=1}^{\infty} g(u_n) =\infty$,   the Hausdorff $f$-measure of the exact set satisfies  $$\H^f( E(Q,R,\phi)) =\infty.$$
\end{thm}


The paper is organized as follows. Section~2 reviews the necessary preliminaries and establishes
several key lemmas. To prove Theorem~\ref{mainthm}, we introduce the quantity
\( G := \limsup_{n\to\infty} g(u_n) \). The proof of Theorem~\ref{mainthm} is then divided
between Sections~3 and~4, corresponding to the cases \( 0 \le G < \infty \) and \( G = \infty \),
respectively. In the final section, we detail concrete applications of our general framework.


\section{Preliminaries}

Let $(X, d)$ be a compact metric space equipped with a non-atomic probability measure $\mu$. We adopt the following notation:
\begin{itemize}

\item A ball centered at a point $ x \in X$ with radius $r$ is the set $B=B(x,r) = \{ y \in X : d(x,y) < r \} $.
  For a ball $B,$   $r(B)$ denotes its radius.
  \item For $c>0$,   $cB$ denotes the concentric  ball  scaled by $c$, i.e., if $B=B(x,r)$, $cB=B(x,cr).$
  \item An annulus centered at $x\in X$ with inner radius $r$ and outer radius $R$ is the set $A(x, r, R)= B(x,R) \setminus \overline{B(x,r)} $.

  \item For a finite set $E$,  $\#E$  denotes its cardinality.

\end{itemize}

\subsection{Hausdorff measure}
This section provides a brief overview of Hausdorff measures. A {dimension function} is a continuous, nondecreasing function $f : \R^+ \to \R^+$ such that $f(r) \to 0$ as $r \to 0$.

Let $F$ be a subset of a metric space $(X,d).$ For $\varepsilon>0,$ a countable collection of balls $\{B_i\}$ with radii $r(B_i)\le\varepsilon$ is called an $\varepsilon$-cover for $F$ if $F\subset \bigcup_{i}B_i$.
 For a dimension function $f,$ we define
 $$\H^f_\varepsilon(F)=\inf \left\{ \sum\limits_{i=1}^\infty f(r_i): \{B_i\} \ \text{is a $\varepsilon$-cover of $F$} \right\}, $$ where the infimum is taken over all such $\varepsilon$-covers of $F.$  The {{Hausdorff $f$-measure $\H^f(F)$ of $F$}} is then given by
 $$\H^f(F):=\lim\limits_{\varepsilon \rightarrow  0} \H^f_{\varepsilon}(F).$$
  In the case that $f(r)=r^s$ $(s\ge 0),$ the measure $\H^f$ is the standard {$s$-dimensional Hausdorff measure}, denoted $\H^s$. Furthermore, the {   Hausdorff dimension  $\hdim F$ of $F$} is defined as
 $$ \hdim  F=\inf \{s\ge0 : \H^s(F)=0\}=\sup \{s\ge0 : \H^s(F)=\infty\}.$$ For further details, we refer to \cite{F14}.

The following mass distribution principle is a standard  method for establishing  lower bounds on   Hausdorff $f$-measure.

\begin{lem}[Mass Distribution Principle, \cite{F14}]\label{p1}
Let $\mu$ be a probability measure supported on a subset $F$ of $(X,d)$. Suppose there are  $a_0>0$ and $r_0>0$ such that
$$\mu(B(x,r))\le a_0 f(r)$$
for every ball $B(x,r)$ with radius $r\le r_0$. 
If $E$ is a subset of $F$ with $\mu(E)>0,$ then $\mathcal{H}^f(F)\ge \mu(E)/{a_0}$.
\end{lem}

\begin{lem}[The $5r$ covering lemma, \cite{H01}]\label{5rcover}
  Every family $\F$ of balls of uniformly bounded diameter in a metric space $(X,d)$ contains a disjoint subfamily $\G$ such that
  $$ \bigcup_{B\in\F}B\subset \bigcup_{B\in\G}5B.$$
\end{lem}

\subsection{General setup}

This section is devoted to the hypotheses on the metric measure space $(X, d, \mu)$ and the system $(Q, R)$. We also establish several key lemmas that will be used in subsequent proofs.
\begin{de}[Regular measure]\label{as21}
Let $(X, d)$ be a metric space, $\delta>0$. A Borel probability measure $\mu$ on $X$ is called $\delta$-regular   if there is $r_0\in (0,1)$ such that the following two conditions are satisfied:

\begin{enumerate}
    \item[(1)] (Ahlfors regularity)  There exist  $0<a<1<b$ such that for any ball $B(x,r)\subset X$ with $r\le r_0$,
\begin{equation}\label{a1}
a r^\delta  \leq \mu(B(x,r)) \leq b r^\delta,
\end{equation}

\item[(2)] (Annular density) For   $0<\epsilon<1$ and $r\le r_0,$ the annulus $A(x, (1-\epsilon)r, r)\subset X$ is nonempty:
$$A(x, (1-\epsilon)r, r)\neq \emptyset.  $$
\end{enumerate}
\end{de}
 \begin{rem} (1)
A $\delta$-regular measure is also annular regular: for   $0<\epsilon<1$ and $0< r< r_0,$
$$\mu( A(x,(1-\epsilon)r, r))\ge\frac{a\epsilon^{\delta}}{3^{\delta}}r^{\delta}.$$
 To see this, pick a point $y$ in the annulus $A\left(x, (1 - \frac{2\epsilon}{3})r, (1 - \frac{\epsilon}{3})r\right)$ and note that the  ball $B(y,\frac{\epsilon}{3}r)$ is contained within the annulus $A(x,(1-\epsilon)r, r)$.




(2) The annular density condition excludes the case where $X$ is the middle-third Cantor set and $\mu$ is the $\frac{\log 2}{\log 3}$-dimensional Hausdorff measure restricted to it.

 (3) Throughout, we  consider only $\delta$-regular measures. Hereafter, we always fix $r_0$ as the constant whose existence  is required in the definition. Similar remarks apply to the various constants in the subsequent definitions, such as the constants $c,\kappa, n_B$ in the definition of a well-distributed system, and the constants $n_0,\lambda_1, \lambda_2$ in the definition of a regular function.
\end{rem}

From now on, we fix  $u=\{u_n\}$ which is a decreasing sequence converging to 0 and satisfying  $\sup_{n} u_n/u_{n+1} < \infty$.
For \( n \le m \), we define
\begin{equation}\label{d1}
J_u(n, m) := \{\xi \in Q : u_{m+1} \le R(\xi) < u_n\}
\qquad
\text{(with the shorthand } J_u(n) := J_u(n, n)\text{)}.
\end{equation}

\begin{de}[Well-distributed system]\label{as22}
Let $\rho: \mathbb{R}^+ \to \mathbb{R}^+$ be a non-decreasing function with $\lim_{r \to 0} \rho(r) = 0$.
We say the system $(Q, R)$ is \emph{well-distributed} in $X$ relative to $(\rho, u)$, if there exist absolute constants $0 < c, \kappa < 1$ such that the following conditions are satisfied:

\begin{enumerate}
    \item[(1)] {(Well-separated condition)} For any two distinct points $\xi, \zeta \in Q$,
    \begin{equation}\label{u2}
    d(\xi, \zeta) \ge c \cdot \min\left\{R(\xi), R(\zeta)\right\}.
    \end{equation}

    \item[(2)]  {(Local ubiquity)} For any ball $B \subset X$ of radius $r \le r_0$, there is $n_B \in \mathbb{N}$ such that for    $n \ge n_B$,
    \begin{equation}\label{u1}
    \mu\left( \bigcup_{\xi \in J_u(n)} B\left(\xi, \rho(u_n)\right) \cap B \right) \ge \kappa \cdot \mu(B).
    \end{equation}
    \end{enumerate}




\end{de}
\begin{rem}
Observe that if the system $(Q, R)$ is well-distributed in $X$ relative to $(\rho,u)$,  it remains well-distributed relative to $(\rho,s)$ for any subsequence $s$ of $u$.

To verify this, it suffices to check that the local ubiquity condition~\eqref{u1} holds for $s$. Since $s$ is a subsequence of $u$, for every $m \in \mathbb{N}$ there exists $n \ge m$ such that $s_{m+1}=u_{n+1}$.  This implies that $u_{n} \le s_{m},$  whence $J_u(n)\subset J_s(m)$ and $\rho(u_{n})\le \rho(s_{m}).$
Consequently,
\[
\mu\left( \bigcup_{\xi \in J_s(m)} B(\xi, \rho(s_m)) \cap B \right)
\ge \mu\left( \bigcup_{\xi \in J_u(n)} B(\xi, \rho(u_n)) \cap B \right)
\ge \kappa \cdot \mu(B).
\]
\end{rem}

\begin{lem}\label{lem1}
  Let $(X,d, \mu)$ be a $\delta$-regular metric measure space. If $(Q,R)$ is a well-distributed system relative to $(\rho,u),$ 
  then for any ball $B\subset X$ with radius $r\le r_0,$ there exists  $\bar{n}_B\in\mathbb{N}$ such that for each $n\ge \bar{n}_B,$ there is  $\bar{Q}_B(n)\subset J_u(n)$ satisfying the following properties.


  (1) For any distinct $\xi,\zeta\in \bar{Q}_B(n),$ $B(\xi,5\rho(u_{n})) \cap B(\zeta,5\rho(u_{n}))=\emptyset;$

  (2) $  \bigcup_{\xi\in \bar{Q}_B(n)} B(\xi, \rho(u_{n}))\subset B;$

  (3)  $\displaystyle   \#\bar{Q}_B(n) \ge \frac{a^2 \kappa }{b^3} \cdot \frac{1}{50^{\delta}}\cdot \frac{\mu(B)}{\rho(u_n)^{ \delta}}.$

\end{lem}
\begin{proof}
Since $\lim_{n}  \rho(u_{n})=0,$ for a ball $B\subset X,$ there exists $\bar{n}_B>  n_{\frac{1}{2}B}$ such that for any $n\ge \bar{n}_B$, $50\rho(u_{n})<\frac{1}{2}r(B)$. 
 We denote by $\tilde{Q}_B(n)$ the collection of all $\xi\in J_u(n)$  with  $B(\xi, 5\rho(u_{n}))\cap\frac{1}{2}B\not=\emptyset$. Hence $ B(\xi, 25\rho(u_{n}))\subset B$ for  $\xi\in \tilde{Q}_B(n)$. 

Considering the family of balls   $\{B(\xi, 5\rho(u_{n})): \xi \in \tilde{Q}_B(n)  \},$ we extract a subset $\bar{Q}_B(n)$ of $\tilde{Q}_B(n)$ according to   Lemma \ref{5rcover} such that the subfamily  $\{B(\xi, 5\rho(u_{n})): \xi \in \bar{Q}_B(n)  \}$ is disjoint and
\begin{equation*}\label{eq5}\bigcup_{\xi  \in J_u(n)}B\left(\xi, \rho(u_{n})\right) \cap \frac{1}{2}B \subset \bigcup_{\xi \in \bar{Q}_B(n)}B(\xi, 25\rho(u_{n}))\subset B. \end{equation*}
Therefore
\begin{align*}
  \mu \left( \bigcup_{\xi \in \bar{Q}_B(n)} B(\xi, 25\rho(u_{n}))\right)&\ge \mu \left(\bigcup_{\xi  \in J_u(n)}B\left(\xi, \rho(u_{n})\right) \cap \frac{1}{2}B\right)\\
   &\overset{\eqref{u1}}{\ge} \kappa \mu\left (\frac{1}{2}B \right)\overset{\eqref{a1}}{\ge}  \frac{a\kappa }{b2^{\delta}}  \mu(B).
\end{align*}
Moreover, we have
$$   \mu \left( \bigcup_{\xi \in \bar{Q}_B(n)} B(\xi, 25\rho(u_{n}))\right) \overset{\eqref{a1}}{\le} \frac{b}{a}\cdot 25^{\delta}\cdot \sum_{\xi \in \bar{Q}_B(n)} \mu \left(  B(\xi, \rho(u_{n}))\right) .$$
Combining these with \eqref{a1} yields
\begin{equation*} \frac{b^2}{a}\cdot 25^{\delta}\cdot \rho(u_n)^{\delta}\cdot \# \bar{Q}_B(n) \ge \mu \left( \bigcup_{\xi \in \bar{Q}_B(n)} B(\xi, 25 \rho(u_{n}))\right) \ge \frac{a\kappa }{b2^{\delta}} \mu(B).\end{equation*}
\end{proof}

A function $h: \mathbb{R}^+ \to \mathbb{R}^+$ is called \textbf{$\ell$-regular} with respect to a sequence $\ell = \{\ell_n\}$ if there exist $n_0 \in \mathbb{N}$ and   $0 < \lambda_1 < \lambda_2 < 1$ such that for all $n \ge n_0$,
\[
\lambda_1 h(\ell_n) \le h(\ell_{n+1}) \le \lambda_2 h(\ell_n),
\]
and $h(x) \le x$ holds for all $x \le \ell_{n_0}$. 

Now recall that $u = \{u_n\}$ is a positive, decreasing sequence with $\lim_n  u_n = 0$ and $\sup_n u_n / u_{n+1} < \infty$. For any syndetic subsequence $s = \{s_i = u_{n_i}\}$ of $u$ (i.e., $\sup_i (n_{i+1} - n_i) < \infty$), the ratio $s_i / s_{i+1}$ remains uniformly bounded, and any $u$-regular function is also $s$-regular.

\begin{lem}\label{lem3}
Let $\phi: \mathbb{R}^+ \to \mathbb{R}^+$ be a non-decreasing function. Suppose that $(Q, R)$ is a well-distributed system relative to $(\rho, u)$ with $\rho$ being $u$-regular. Then there exist   $\kappa_1, \kappa_2 > 0$ and   $\bar{n}_0 \in \mathbb{N}$ such that for all $n \ge \bar{n}_0$,
\begin{equation}\label{eq7}
\kappa_1 (u_n)^{-\delta} \le \# J_u(n) \le \kappa_2 (u_n)^{-\delta}.
\end{equation}
Moreover, the following equivalence holds for the convergence of the series:
\begin{equation}\label{eq8}
\sum_{\xi \in Q} \left( \phi(R(\xi)) \right)^\delta < \infty \quad \Longleftrightarrow \quad \sum_{n=1}^\infty \left( \frac{\phi(u_n)}{u_n} \right)^\delta < \infty.
\end{equation}
\end{lem}
\begin{proof} We first observe that the equivalence \eqref{eq8} follows directly from the cardinality estimate \eqref{eq7} and the monotonicity of $\phi$. It therefore suffices to establish \eqref{eq7}.

To this end, fix a ball $B_0 \subset X$ with radius at most $r_0$. Since $\sup_{n} u_n / u_{n+1} < \infty$, there exists   $\tilde{c} < 1$ such that
\begin{equation}\label{eq6}
\frac{u_n}{u_{n+1}} < \frac{1}{\tilde{c}} \quad \text{for all } n \ge 1.
\end{equation}
Let $\bar{n}_0 = \max\{n_{B_0}, n_0\}$.

On the one hand, the $u$-regularity of $\rho$ and the local ubiquity property imply that for all $n \ge \bar{n}_0$,
\begin{align*}
   b(u_n)^\delta\# J_u(n) \ge \mu\left( \bigcup_{\xi \in J_u(n)} B(\xi,u_n) \cap B_0 \right) &\ge     \mu\left( \bigcup_{\xi \in J_u(n)} B(\xi,\rho(u_n)) \cap B_0 \right)\\
    & \ge \kappa \mu(B_0),
  \end{align*}
which yields the lower bound in \eqref{eq7} with $\kappa_1 = \kappa b^{-1} \mu(B_0)$.

On the other hand, it follows from the separation condition \eqref{u2} that for any two distinct points $\xi, \zeta \in J_u(n)$,
\[
d(\xi, \zeta) \ge c u_{n+1}.
\]
Consequently, the balls $B\left(\xi, \frac{c}{2} u_{n+1}\right)$ are pairwise disjoint. Using  \eqref{a1} and   \eqref{eq6}, a standard volume argument gives the upper bound in \eqref{eq7} with $\kappa_2 =  {2^\delta}/{a (c \tilde{c})^\delta}$.
  \end{proof}

In what follows, we always write $c_l = 1 - 2^{-l}$ for every integer $l \ge 0$. For each $\xi \in Q$ and $l \ge 1$, we define the annulus
\[
A_l(\xi) := A\left(\xi, c_l \phi(R(\xi)), \phi(R(\xi))\right).
\]

\begin{lem}\label{lem2}
Let \((X,d, \mu)\) be a \(\delta\)-regular metric measure space, and let \((Q,R)\) be a well-distributed system relative to \((\rho,u)\) with \(\rho\) being \(u\)-regular. Let \(\phi:\mathbb{R}^{+}\rightarrow \mathbb{R}^{+}\) be a non-decreasing function such that
\[
\lim_{r\rightarrow 0} \frac{\phi(r)}{r} = 0 \quad \text{and} \quad \sum_{\xi\in Q} (\phi(R(\xi)))^{\delta} < \infty.
\]
Then there exists \(m_0 \in \mathbb{N}\) such that for any   \(l \ge 1\),  \(n > m_0\), and  \(\xi \in J_u(n)\), the following holds: for every ball
\[
B = B\left(x_0, \frac{1 - c_l}{3} \phi(R(\xi))\right) \subset A_l(\xi),
\]
there is   \(m_1(B)\ge n\) such that for all \(m > m_1(B)\), one can find   \(Q_B(m) \subset J_u(m)\) satisfying:

(1) For any distinct \(\gamma, \gamma' \in Q_B(m)\),
  $  B(\gamma, 5\rho(u_m)) \cap B(\gamma', 5\rho(u_m)) = \emptyset.$

   (2) 
     $  \bigcup_{\gamma \in Q_B(m)} B(\gamma, \rho(u_m)) \subset B.$

    (3) 
     $\displaystyle \# Q_B(m) \ge \frac{3a^2 \kappa}{4b^3} \cdot \frac{1}{50^{\delta}} \cdot \frac{\mu(B)}{\rho(u_m)^{\delta}}.$


 (4) For any \(x \in \bigcup_{\gamma \in Q_B(m)} A_{l}(\gamma)\) and   \(\eta \in J_u(n,m)\),  $
   d(x, \eta) \ge c_l \phi(R(\eta)).$

   \end{lem}

\begin{proof}
Let $\tilde{c}$ be the constant from \eqref{eq6}. Since $\lim_{r \to 0} \phi(r)/r = 0$, for any integer $l \ge 1$ there exists $\epsilon_0 > 0$ such that
\begin{equation}\label{c1}
\phi(r) < \min\left\{ r_0, \frac{\tilde{c} \cdot c \cdot (1 - c_l)}{4} r \right\} \quad \text{for all } r \le \epsilon_0.
\end{equation}
Moreover, choose $m_0 \ge n_0$ sufficiently large such that for all $k > m_0$,
\begin{equation}\label{c2}
u_k \le \epsilon_0 \quad \text{and} \quad \sum_{i \ge k} \left( \frac{\phi(u_i)}{u_i} \right)^\delta < \frac{\kappa}{4} \cdot \frac{a^5}{b^5} \cdot \left( \frac{\lambda_1 \cdot \tilde{c}^2 \cdot c^2 \cdot (1 - c_l)}{10^6} \right)^\delta.
\end{equation}
The existence of such an $m_0$ is guaranteed by the convergence of the series $\sum_{\xi \in Q} (\phi(R(\xi)))^\delta$, in conjunction with Lemma \ref{lem3} and the fact that $\lim_{k \to \infty} u_k = 0$.

Fix a ball \( B \subset A_l(\xi) \), where \( \xi \in J_u(n) \) and \( n > m_0 \). Since \( R(\xi) < u_n \), \( r(B) < r_0 \). We may  apply Lemma~\ref{lem1} to the ball \( B \). Let \( \bar{n}_B \) be the constant provided by that lemma, and set
\(
m_1(B) = \max\{ n, n_0, \bar{n}_B \}.
\)
Then for any \( m > m_1(B) \), there exists   \( \bar{Q}_B(m) \subset J_u(m) \) satisfying conditions (1)-(3) of Lemma~\ref{lem1} (with \( m \) in place of \( n \)).
Define
\[
\mathrm{Bad}_B(m) = \Big\{ \gamma \in \bar{Q}_B(m) :  A_{l}(\gamma) \cap B(\eta, c_l \phi(R(\eta)))\ne \emptyset\text{ for some }
 \eta \in    J_u(n,m) \Big\}
\]
and set
\[
Q_B(m) = \bar{Q}_B(m) \setminus \mathrm{Bad}_B(m).
\]
To verify that \( Q_B(m) \) satisfies the conclusion of the lemma, it suffices to show that
\[
\# \mathrm{Bad}_B(m) < \frac{1}{4} \# \bar{Q}_B(m).
\]
The proof of this estimate will be carried out in the following steps.

\smallskip
\underline{\em Claim 1.} We have
\begin{equation}\label{S1}
\mathrm{Bad}_{B}(m) \subset \mathrm{Bad}'_{B}(m),
\end{equation}
where
\[
\mathrm{Bad}'_{B}(m) =\{ \gamma \in \bar{Q}_B(m) : \gamma \in B(\eta, \phi(R(\eta))) \text{ for some }
  \eta \in J_u(n,m)
\}.
\]

Let \(\gamma \in \bar{Q}_B(m)\), and suppose there exists \(\eta \in J_u(n,m)\) such that
\[
B(\eta, c_l \phi(R(\eta))) \cap A_{l}(\gamma) \ne \emptyset.
\]
Clearly, $\gamma \neq \eta$.
Since \(m > m_0\), inequalities \eqref{c1} and \eqref{c2} imply
\begin{equation}\label{j1}
\frac{\phi(R(\gamma))}{R(\gamma)} < \tilde{c} \cdot c \cdot (1 - c_l).
\end{equation}

From \eqref{d1} and \eqref{eq6}, we have \(R(\gamma) < u_m < \tilde{c}^{-1} R(\eta)\), and   \begin{equation}\label{j8}\tilde{c}\cdot c \cdot R(\gamma)< c\min\{R(\eta), R(\gamma)\}\overset{\eqref{u2}}{\le} d(\eta,\gamma). \end{equation}

On the other hand, selecting  a point \(z\) in $B(\eta, c_l \phi(R(\eta))) \cap A_{l}(\gamma)$, we have
\begin{align}\label{j2}
d(\eta, \gamma) \le d(\eta, z) + d(z, \gamma) \nonumber
&< c_l \phi(R(\eta)) + \phi(R(\gamma)) \nonumber\\
&\overset{\eqref{j1}}{<} c_l \phi(R(\eta)) + \tilde{c} \cdot c \cdot (1 - c_l) R(\gamma).
\end{align}

Combining \eqref{j8} and \eqref{j2} gives
\begin{equation}\label{j3}
R(\gamma) < (\tilde{c} \cdot c)^{-1} \phi(R(\eta)).
\end{equation}

Substituting \eqref{j3} back into \eqref{j2}, we obtain
\[
d(\eta, \gamma) < c_l \phi(R(\eta)) + (1 - c_l) \phi(R(\eta)) = \phi(R(\eta)),
\]
which shows that \(\gamma \in B(\eta, \phi(R(\eta)))\), and hence establishes \eqref{S1}.

\smallskip
\underline{\em Claim 2.} For $n\le t\le m,$ the set $$ S_{t}(B):=\{\eta \in Q: u_{t+1}\le R(\eta)<u_t~\text{and}~B(\eta,\phi(R(\eta)))\cap B \neq \emptyset  \} $$ satisfies the cardinality estimate
\begin{equation}\label{c4}
  \# S_{t}(B)\le \frac{b}{a^2}\cdot \left(\frac{24}{  c\cdot (1-c_l)\cdot\tilde{c}^2} \right)^{\delta}\cdot (u_t)^{-\delta} \cdot \mu(B).
\end{equation}

Since  $d(\eta, \eta')\ge c\cdot u_{t+1}$ for distinct $\eta, \eta' \in S_{t}(B),$ the balls $\{B(\eta,\frac{c\cdot \tilde{c} }{4}u_{t+1}) \}_{\eta \in S_{t}(B)}$ are pairwise disjoint. For $n\le t\le m,$ we prove that
\begin{equation}\label{c3}
  \bigcup_{\eta \in S_{t}(B)} B\left(\eta,\frac{\tilde{c} \cdot c  }{4}u_{t+1} \right)\subset B(\xi,2\phi(R(\xi))).
\end{equation} 
Fix $\eta \in S_{t}(B).$ From $t> m_0$, \eqref{c1}, and \eqref{c2}, we conclude
\begin{equation}\label{j5}
  \phi(R(\eta))<\frac{\tilde{c}\cdot c}{4}R(\eta).
\end{equation} Note that \begin{equation}\label{j4}
  R(\eta)< u_n< \tilde{c}^{-1}R(\xi).
\end{equation} Since $B(\eta, \phi(R(\eta)))\cap B\neq \emptyset$ and $B\subset A_l(\xi), $ there exists $\xi' \in B(\eta, \phi(R(\eta)))\cap B(\xi, \phi(R(\xi)) ).$ It follows from \eqref{u2},
\eqref{j5} and \eqref{j4} that
\begin{align}\label{j7}
  \tilde{c}\cdot c \cdot R(\eta)&\le d(\xi,\eta)\le d(\xi,\xi')+d(\xi',\eta) \nonumber\\
  &\le \phi(R(\xi))+\phi(R(\eta)) \le \phi(R(\xi))+\frac{\tilde{c}  \cdot c}{4}R(\eta).
\end{align}Thus,
\begin{equation}\label{j6}
  \frac{3 }{4}\cdot   \tilde{c} \cdot c \cdot R(\eta)\le \phi(R(\xi)).
\end{equation} Take $\tilde{\eta }\in B(\eta, \frac{  \tilde{c}\cdot c}{4}u_{t+1}).$ By \eqref{j7}, \eqref{j6} and the fact that $u_{t+1}\le R(\eta),$ we get that
\begin{align*}d(\xi,\tilde{\eta}) \le d(\xi,\eta)+d(\eta,\tilde{\eta})&\le \phi(R(\xi))+\frac{\tilde{c}\cdot c}{4}R(\eta)+\frac{  \tilde{c}\cdot c}{4}R(\eta)\\
&\le 2 \phi(R(\xi)).\end{align*} This establishes \eqref{c3}. Applying \eqref{a1} and \eqref{eq6}, we have
\begin{align*}
   a\left( \frac{  c \cdot  \tilde{c}^2\cdot u_{t}}{4}\right)^{\delta} \#S_{t}(B)&\le \sum_{\eta \in S_{t}(B)} \mu \left(B\left(\eta, \frac{  \tilde{c} \cdot c}{4}u_{t+1}\right) \right)\le \mu(B(\xi,2\phi(R(\xi)))\\
  &\le b\cdot 2^{\delta} \cdot \phi^{\delta}(R(\xi))\le  \frac{b}{a} \cdot  \frac{6^{\delta}}{(1-c_l)^{\delta}} \cdot \mu(B),
\end{align*} which gives \eqref{c4}.

\smallskip
\underline{\em Claim 3.} We have that
\begin{equation}\label{c5}
  \bigcup_{\gamma \in  \mathrm{Bad}'_{B}(m) } B(\gamma,r_1) \subset \bigcup_{n\le t\le m} \bigcup_{ \eta\in S_{t}(B) }B(\eta,2 \phi(R(\eta))),
\end{equation}
where $r_1=\frac{c}{4} \min\{\rho(R(\gamma)): \gamma \in  \mathrm{Bad}'_{B}(m)\}.$

Take $\gamma' \in B(\gamma,r_1)$ for some $\gamma \in  \mathrm{Bad}'_{B}(m).$ By the definition of $ \mathrm{Bad}'_{B}(m)$ and the fact  $\gamma \in B,$ there exists $\eta \in S_{t}(B)$ with $n\le t\le m$ such that $\gamma \in B(\eta, \phi(R(\eta))).$ Combining \eqref{u2}, \eqref{c1} with \eqref{c2}, we get
 \begin{equation*}
   c\min\{R(\gamma),R(\eta)\}\le d(\gamma,\eta)<\phi(R(\eta))<\frac{c}{4}\cdot R(\eta).
 \end{equation*} Thus, $$\min\{R(\gamma),R(\eta)\}=R(\gamma) .$$ Then, employing the monotonicity  and  $u$-regular property of $\rho,$  we obtain that
 $$ r_1 \le \frac{c}{4} \rho(R(\gamma)) \le  \frac{c}{4}  R(\gamma) \le \frac{1}{4}d(\gamma, \eta)<\frac{1}{4}\phi(R(\eta)).$$ Therefore,
 $$d(\gamma', \eta)\le d(\gamma',\gamma)+d(\gamma,\eta)\le r_1+\phi(R(\eta))\le 2\phi(R(\eta)). $$

\smallskip
\underline{\em Claim 4.} We conclude that
\[
\# \mathrm{Bad}_B(m) < \frac{1}{4} \# \bar{Q}_B(m).
\]

It follows  from \eqref{a1}, \eqref{u2} and the $u$-regular property of $\rho$ that for any $\gamma\in \mathrm{Bad}'_{B}(m),$  $$ \mu (B(\gamma,r_1))\ge a\left(\frac{c\cdot\lambda_1 \cdot \rho(u_{m})}{4}\right)^{\delta},$$ and that
the balls $\{ B(\gamma,r_1) \}_{\gamma \in  \mathrm{Bad}'_{B}(m) }$ are pairwise disjoint. Finally, using \eqref{a1}, \eqref{c4}, \eqref{c5} and the monotonicity of $\phi,$ we can assert that
\begin{align*}
  a\left(\frac{c\cdot\lambda_1 \cdot \rho(u_{m})}{4}\right)^{\delta} \#  \mathrm{Bad}'_{B}(m) &\le \sum_{\gamma \in \mathrm{Bad}'_{B}(m) } \mu (B(\gamma,r_1)) \\
  & \le \sum_{m \le t\le n} \sum_{\eta\in S_{t}(B)} \mu(B(\eta,2 \phi(R(\eta))) )  \\
  & \le \frac{b^2}{a^2} \left(\frac{48}{ c \cdot (1-c_l)\cdot \tilde{c}^2 } \right)^{\delta}\mu(B) \sum\limits_{m \le t\le n}\left(\frac{\phi(u_t)}{u_t}\right)^{\delta}.
\end{align*}
Combining \eqref{c2} with the fact that $n>m_0,$ we obtain
$$ \#  \mathrm{Bad}_{B}(m)\le  \#  \mathrm{Bad}'_{B}(m)< \frac{\kappa}{4} \cdot \frac{a^2}{b^3\cdot 50^{\delta}}\cdotp \frac{\mu(B)}{\rho(u_m)^{\delta}} \le \frac{1}{4}\#\bar{Q}_B(m).  $$\end{proof}

\section{Proof of Theorem \ref{mainthm}: $0\le G<\infty$}
 Fix a ball $B_0\subset X$ with radius at most $r_0$. Theorem \ref{mainthm} thus follows once we show
 \begin{equation}\label{g1}
   \H^f(B_0\cap E(Q,R,\phi)) = \infty .
 \end{equation}

To prove \eqref{g1}, we construct, for any sufficiently large constant \(\eta > 0\),
a Cantor subset \(K_{\eta} \subseteq B_0 \cap E(Q,R,\phi)\) together with a probability measure
\(\nu\) supported on \(K_{\eta}\) that satisfies the scaling property
\begin{equation}\label{g2}
\nu(B(x,r)) \le \frac{6^\delta}{\alpha} \frac{f(r)}{\eta}
\end{equation}
for every ball \(B(x,r)\) with sufficiently small radius \(r\).
Here, \(\alpha\) is the constant defined in \eqref{defalpha}, which is independent of both
\(B(x,r)\) and \(\eta\).
Then, by the Mass Distribution Principle (Proposition~\ref{p1}), we obtain
\[
\mathcal{H}^f\bigl(B_0 \cap E(Q,R,\phi)\bigr) \ge \mathcal{H}^f(K_{\eta})
\ge \eta \frac{\alpha}{6^\delta}.
\]
Letting \(\eta \to \infty\) gives \eqref{g1} and hence proves Theorem~\ref{mainthm}.


We now proceed to the construction of the Cantor set \(K_{\eta}\), which is to support a measure
\(\nu\) satisfying the required property. This construction relies on the following lemma.

\begin{lem}\label{as}In establishing Theorem \ref{mainthm} for the case that $0\le G<\infty,$  we can assume that the function $\rho$ is  $u$-regular with an arbitrarily small constant $\lambda_2.$
\end{lem}
\begin{proof}
 We recall  the following simple facts:

(i) If the system $(Q, R)$ is well-distributed in $X$ relative to $(\rho,u)$  it is also well-distributed relative to $(\rho,s)$  for any subsequence $s$ of $u;$

(ii) For any syndetic subsequence $ s $ of $u,$ 
 we have $\sup {s_n}/{s_{n+1}}<\infty,$ and the $ u $-regularity of $ \rho $ implies its $ s $-regularity; 


 (iii) If $ G $ is finite, then $\limsup_{n \to \infty} g(s_n) < \infty$ for any subsequence $ s $ of $u.$

 To establish the desired result, it suffices to prove that, if $\rho$ is $u$-regular with   $\lambda_2<1,$ then  for any positive  integer $m$ one can find a subsequence $\{u_{n_i}\}$ of $u$ such that
 $$\sup (n_{i+1}-n_i) < \infty ,~ \rho(u_{n_{i+1}})<\lambda_2^m \rho(u_{n_i}), \text{ and }  \sum_{i=1}^{\infty} g(u_{n_i})=\infty.$$
 To this end, we take  $l_i\in\{m(i-1)+1,m(i-1)+2,\ldots,mi\}$ such that $$g(u_{l_i})=\max\{g(u_n): m(i-1)<n\le mi \}.$$ Then
  \[  \infty=\sum_{n=1}^{\infty}g(u_n)  \le m \sum_{i=1}^{\infty}g(u_{l_i})=m \sum\limits_{i=1}^{\infty}g(u_{l_{2i-1}}) +m \sum\limits_{i=1}^{\infty}g(u_{l_{2i}}).
\]
Without loss of generality, we may assume that
   $ \sum_{i=1}^{\infty}g(u_{l_{2i}})=\infty. $ Putting $n_i=l_{2i}$, we have that  $m<n_{i+1} -n_{i} < 3m,$ and $\rho(u_{n_{i+1}})<\lambda_2^{m}\rho(u_{n_{i}})$, which  completes the proof. \end{proof}


\subsection{Cantor subset construction}
In this section, the Cantor set \( K_{\eta} \) is constructed recursively.
We begin by setting \( \mathcal{K}_0 = \{B_0\} \) and \( K_0 = B_0 \).
The first two levels are treated in detail; the construction of the remaining levels proceeds in an analogous manner.


Recall that we are in the regime \( 0 \le G < \infty \).
We define
\[
G^* := \max\{2\mu(B_0),\, \sup g(u_n)\} \in (0,\infty),
\]
and fix a real number \( \eta \) sufficiently large that
\begin{equation}\label{nla1}
\eta > \frac{16b}{a} \, G^*.
\end{equation}
Next, we set
\begin{equation}\label{defalpha}
a_1 := \frac{3a^3 \kappa}{4b^3} \cdot \frac{1}{50^{\delta}} \in (0,1),\qquad
\alpha := \frac{a_1 a \lambda_1^\delta}{64b } < 1.
\end{equation}
In view of Lemma~\ref{as}, we may assume the existence of \( N_0 \in \mathbb{N} \) such that
\[
\rho(u_{n+1}) \le \lambda_2 \, \rho(u_n) \qquad \text{for all } n \ge N_0,
\]
where the constant \( \lambda_2 \) satisfies
\begin{equation}\label{as1}
0 < \lambda_2 < \left( \frac{a}{a+3^\delta 8b} \right)^{\! 1/\delta}.
\end{equation}

\subsubsection{Construction of the first level $K_1$}\label{sec1}
Let $\bar{n}_{B_0}$ and $m_0$ be the constants from Lemmas \ref{lem1} and   \ref{lem2}, respectively. Choose $n_1>\max\{ \bar{n}_{B_0},m_0,N_0\}$ sufficiently large so that
\begin{equation} \label{cc2}
  \phi(u_n) <\rho(u_n) \quad \text{for all}~ n\ge n_1,
\end{equation}
\begin{equation}\label{cc1}
g(u_{n_1}) \le  G^*<\frac{a}{16b}  \frac{\eta }{\alpha } \left(\frac{3}{(1-c_1)r(B_0)}\right)^{\delta} ,
\end{equation}
\begin{equation}\label{cc3}
\frac{f(\phi(u_{n_1}))}{ \phi(u_{n_1})^{\delta} } >\frac{\eta }{\alpha }\left(\frac{3}{r(B_0)}\right)^\delta.
\end{equation}
Note that (\ref{cc2}) holds by the assumption \(\limsup_{n\to\infty} \phi(u_n)/\rho(u_n)=0\);
(\ref{cc1}) follows from the fact that \(g(u_n) < G^*\) for all \(n\) together with \eqref{nla1}; and
(\ref{cc3}) is a consequence of the monotonicity of \(f(r)/r^\delta\).

Define \(k_1(B_0)\) to be the unique integer satisfying
\begin{equation}\label{cc4}
\frac{2b}{a} \frac{\alpha}{\eta}
\left(\frac{(1-c_1)r(B_0)}{3}\right)^{\!\delta}
~\sum_{i=0}^{k_1(B_0)-1} g(u_{n_1+i}) < \frac{1}{4}
\end{equation}
and
\begin{equation}\label{cc5}
\frac{2b}{a} \frac{\alpha}{\eta}
\left(\frac{(1-c_1)r(B_0)}{3}\right)^{\!\delta}
~\sum_{i=0}^{k_1(B_0)} g(u_{n_1+i}) \ge \frac{1}{4}.
\end{equation}
That \(k_1(B_0) \ge 1\) follows directly from \eqref{cc1}.


For each \(n \ge n_1 \;(\ge \bar{n}_{B_0})\), Lemma~\ref{lem1} guarantees the existence of a subset
\(\bar{Q}_{B_0}(n) \subset J_u(n)\) satisfying conditions (1)--(3).
By the annular density property, for every \(\xi \in \bar{Q}_{B_0}(n)\) we may select a point
\(x_\xi \in A_1(\xi)\) such that
\begin{equation}\label{s1}
B\biggl(x_\xi, \frac{1-c_1}{3} \phi(R(\xi))\biggr) \subset A_1(\xi) \subset B(\xi, \phi(u_n)).
\end{equation}
Denote by \(C'_1(B_0,n)\) the collection of all such centers \(x_\xi\) as \(\xi\) ranges over
\(\bar{Q}_{B_0}(n)\). Then
\[
\# C'_1(B_0,n) = \# \bar{Q}_{B_0}(n) \ge a_1 \left( \frac{r(B_0)}{\rho(u_n)} \right)^{\!\delta}.
\]
By discarding some points if necessary, we obtain a subcollection
  \(C_1(B_0,n)\)  of \(C'_1(B_0,n)\) satisfying
\[
\# C_1(B_0,n) = \left\lfloor a_1 \left( \frac{r(B_0)}{\rho(u_n)} \right)^{\!\delta} \right\rfloor,
\]
where \(\lfloor z \rfloor\) denotes the greatest integer \(\le z\). Consequently,
\begin{equation}\label{num1}
\frac{a_1}{2} \left( \frac{r(B_0)}{\rho(u_n)} \right)^{\!\delta}
\le \# C_1(B_0,n) \le a_1 \left( \frac{r(B_0)}{\rho(u_n)} \right)^{\!\delta}.
\end{equation}

For each center \(x \in C_1(B_0,n)\), let \(\xi_x \in \bar{Q}_{B_0}(n)\) denote the unique point
\(\xi\) such that \(x \in A_1(\xi)\).
From \eqref{cc2} and \eqref{s1} we obtain
\[
B\bigl(x, \phi(R(\xi_x))\bigr) \subset B(x,\phi(u_n))
\subset B(x,\rho(u_n)) \subset B(\xi_x, 2\rho(u_n)).
\]
Thus the balls \(B(x, 3\rho(u_n))\) with \(x \in C_1(B_0,n)\) are pairwise disjoint,
since the larger balls \(B(\xi_x, 5\rho(u_n))\) with \(\xi_x \in \bar{Q}_{B_0}(n)\) are themselves disjoint.

We now provide the explicit construction of the sub-levels \( \K_1(B_0, n_1+i) \) for \( i = 0, 1, \ldots, K_0(B) \).

\smallskip
\textbullet \ {\em{The sub-level $\K_1(B_0, n_1).$}}
We define \[G_1(B_0,n_1):=C_1(B_0,n_1).\]
The first sub-level is defined to be
$$ K_1(B_0, n_1):=\bigcup_{B \in \K_1(B_0, n_1)} B ,$$
 where
$$  \K_1(B_0, n_1):=\left\{B\left(x, \frac{1-c_1}{3} \phi(R(\xi_x))\right):  x\in G_1(B_0,n_1) \right\}.$$

\smallskip
\textbullet \ {\em{The sub-level $\K_1(B_0, n_1+1).$}}


Set
\[
h(n_1) := \left( \frac{\alpha}{\eta}
\left( \frac{(1-c_1)r(B_0)}{3} \right)^{\!\delta}
f(\phi(u_{n_1})) \right)^{\!1/\delta}.
\]
By \eqref{cc3}, we have
\[
B\bigl(x, \tfrac{1-c_1}{3}\phi(R(\xi_x))\bigr)
\subset B\bigl(x, (1-c_1)\phi(u_{n_1})\bigr)
\subset B(x, h(n_1)).
\]
We therefore refer to \( B(x, h(n_1)) \) as a \textit{thickening} of the ball
\( B\bigl(x, \tfrac{1-c_1}{3}\phi(R(\xi_x))\bigr) \).
Moreover, \eqref{cc1} implies
\[
B(x, h(n_1)) \subset B(x, \rho(u_{n_1})).
\]
Now set
\[
T_1(B_0,n_1) := \{ B(x, h(n_1)) : x \in G_1(B_0,n_1) \}.
\]
Then \(\# T_1(B_0,n_1) = \# G_1(B_0,n_1)\).
Clearly, the thickenings in \(T_1(B_0,n_1)\) are pairwise disjoint,
since the balls \(B(x, 3\rho(u_{n_1}))\) with \(x \in G_1(B_0, n_1)\) are themselves disjoint.

Our goal is to obtain a family of balls that are sufficiently separated from the previous sub-level.
To this end, we remove certain balls from the collection
\[
\left\{ B\bigl(x, \tfrac{1-c_1}{3} \phi(R(\xi_x))\bigr) : x \in C_1(B_0, n_1+1) \right\}.
\]
More precisely, define
\[
U_1(B_0, n_1+1) := \bigl\{ y \in C_1(B_0, n_1+1) :
B(y, 3\rho(u_{n_1+1}))\text{ intersects some thickening in }    T_1(B_0,n_1)  \bigr\},
\]
and set
\[
G_1(B_0, n_1+1) := C_1(B_0, n_1+1) \setminus U_1(B_0, n_1+1).
\]
We claim that
\[
\# G_1(B_0, n_1+1) \ge \frac{1}{2} \, \# C_1(B_0, n_1+1).
\]

For a fixed \( x \in G_1(B_0, n_1) \), we define
\[
N_x := \#\{ y \in C_1(B_0, n_1+1) : B(y, 3\rho(u_{n_1+1})) \cap B(x, h(n_1)) \neq \varnothing \}.
\]
The proof of the claim proceeds by considering two cases.

\emph{Case 1: } \( 3\rho(u_{n_1+1}) < h(n_1). \)

Assume that \( B(y, 3\rho(u_{n_1+1})) \) intersects  \( B(x, h(n_1))\). The condition \( 3\rho(u_{n_1+1}) < h(n_1) \) then implies the inclusion
\[
B(y, 3\rho(u_{n_1+1})) \subset B(x, 3h(n_1)).
\]
On the other hand, the balls \( B(y, 3\rho(u_{n_1+1})) \) for \( y \in C_1(B_0, n_1+1) \) are pairwise disjoint. Consequently, we have
\[
b \cdot 3^{\delta} \cdot h(n_1)^{\delta} \ge \mu\bigl(B(x, 3h(n_1))\bigr)
\ge N_x \cdot a \cdot 3^{\delta} \cdot \rho(u_{n_1+1})^{\delta}.
\]
Hence,
\[
N_x \le \frac{b}{a} \left( \frac{h(n_1)}{\rho(u_{n_1+1})} \right)^{\!\delta}.
\]
Applying \eqref{cc1} and \eqref{num1}, we conclude that
\[
\begin{aligned}
\#U_1(B_0, n_1+1)
&\le \sum_{x \in G_1(B_0, n_1)} N_x
\le \#T_1(B_0, n_1) \cdot \frac{b}{a} \left( \frac{h(n_1)}{\rho(u_{n_1+1})} \right)^{\!\delta} \\[4pt]
&\le \frac{2b}{a} \cdot \frac{\alpha}{\eta} \cdot
\left( \frac{(1-c_1)r(B_0)}{3} \right)^{\!\delta} \cdot
g(u_{n_1}) \cdot \#C_1(B_0, n_1+1) \\[4pt]
&< \frac{1}{4} \, \#C_1(B_0, n_1+1).
\end{aligned}
\]

\emph{Case 2: } \( 3\rho(u_{n_1+1}) \ge h(n_1). \)

If \( B(y, 3\rho(u_{n_1+1})) \) intersects \( B(x, h(n_1)) \), the condition \( 3\rho(u_{n_1+1}) \ge h(n_1) \) then implies that
\[
B(y, 3\rho(u_{n_1+1})) \subset B(x, 9\rho(u_{n_1+1})).
\]
It follows that
\[
b \cdot 9^{\delta} \cdot \rho(u_{n_1+1})^{\delta}
\ge \mu\bigl(B(x, 9\rho(u_{n_1+1}))\bigr)
\ge N_x \cdot a \cdot 3^{\delta} \cdot \rho(u_{n_1+1})^{\delta},
\]
and hence
\[
N_x \le \frac{b}{a} \cdot 3^{\delta}.
\]

By \eqref{as1} and \eqref{num1}, we obtain
\[
\begin{aligned}
\# U_1(B_0, n_1+1)
&\le \sum_{x \in G_1(B_0, n_1)} N_x
\le \#T_1(B_0, n_1) \cdot \frac{b}{a} \cdot 3^{\delta} \\[4pt]
&\le \frac{2b \, 3^{\delta}}{a}
\left( \frac{\rho(u_{n_1+1})}{\rho(u_{n_1})} \right)^{\!\delta}
\# C_1(B_0, n_1+1) \\[4pt]
&\le \frac{2b \, 3^{\delta}}{a} \, \lambda_2^{\delta}
\# C_1(B_0, n_1+1)
< \frac{1}{4} \, \# C_1(B_0, n_1+1).
\end{aligned}
\]

Combining the two cases gives
\(
\#U_1(B_0, n_1+1) < \frac{1}{2} \, \# C_1(B_0, n_1+1).
\)
Hence,
\[
\#G_1(B_0, n_1+1) \ge \frac{1}{2} \, \# C_1(B_0, n_1+1).
\]

The second sub-level is now defined as
\[
K_1(B_0, n_1+1) := \bigcup_{B \in \mathcal{K}_1(B_0, n_1+1)} B,
\]
where
\[
\mathcal{K}_1(B_0, n_1+1) :=
\left\{ B\bigl(x, \tfrac{1-c_1}{3} \phi(R(\xi_x))\bigr) : x \in G_1(B_0, n_1+1) \right\}.
\]
By construction, we have \( K_1(B_0, n_1) \cap K_1(B_0, n_1+1) = \varnothing \).

\medskip

When $k_1(B_0)\ge 2$, we  now construct the general sub-levels.

\textbullet \ {\em{The sub-level $\K_1(B_0, n_1+i)$ with $2\le i \le k_1(B_0).$ }}

Assuming that, for $1\le j \le i-1,$  the sub-levels
$$\K_1(B_0, n_1+j):=\left\{B\left(x, \frac{1-c_1}{3} \phi(R(\xi_x))\right): x\in G_1(B_0,n_1+j) \right\}$$  are already constructed, we turn to the construction of the next  sub-level $\K_1(B_0, n_1+i).$

Set
$$ h(n_1+i-1)= \left(  \frac{\alpha  }{\eta  } \left(\frac{(1-c_1)r(B_0)}{3}\right)^{\delta} f(\phi(u_{n_1+i-1}))\right)^{1/\delta}.$$
By \eqref{cc3} and the fact that $f(r)/r^{\delta}$ is decreasing, we have $$B\left(x, \frac{1-c_1}{3}\phi(R(\xi_x))\right)\subset B\left(x,(1-c_1)\phi(u_{n_1+i-1})\right)\subset B(x,h(n_1+i-1)).$$ Moreover, combining \eqref{cc1} with the definition of $G^*,$ we conclude that
$$ B(x,h(n_1+i-1))\subset B(x,\rho(u_{n_1+i-1})).$$
Define
$$ T_1(B_0, n_1+i-1):=\{ B(x, h(n_1+i-1)): x\in G_{1}(B_0,n_1+i-1)\}.$$ 
To ensure sufficient separation, we introduce \( U_1(B_0, n_1+i) \), the set of \( y \in C_1(B_0, n_1+i) \) such that
\( B(y, 3\rho(u_{n_1+i})) \) intersects some thickening \( B(x, h(n_1+j)) \) in
\( T_1(B_0, n_1+j) \) with \( j = 0, 1, \ldots, i-1 \). Define
\[
G_1(B_0, n_1+i) := C_1(B_0, n_1+i) \setminus U_1(B_0, n_1+i).
\]
We claim that
\[
\# G_1(B_0, n_1+i) \ge \frac{1}{2} \, \# C_1(B_0, n_1+i).
\]

Following the same case analysis as for $i=1$, we have



\begin{align*}
\#U_1(B_0, n_1+i)\le &\sum \limits_{j \in J_1 }\#T_1(B_0, n_1+j) \cdot  \frac{b}{a}\cdot \left( \frac{h(n_1+j)}{\rho(u_{n_1+i})} \right)^{\delta}  \\
&+\sum \limits_{j \in J_2 }\#T_1(B_0,n_1+j) \cdot\frac{b}{a}\cdot 3^{\delta},
\end{align*}
where $J_1=\{0\le j \le i-1: 3\rho(u_{n_1+i})<h(n_1+j) \}$, $J_2=\{0\le j \le i-1: 3\rho(u_{n_1+i})\ge h(n_1+j) \}$.
By \eqref{cc4} and \eqref{num1}, the contribution of the first summation is:
\begin{align*}
  &\le \sum_{j \in J_1} \frac{2b}{a} \frac{\alpha}{\eta} \left(\frac{(1-c_1)r(B_0)}{3}\right)^\delta g(u_{n_1+j}) \, \#C_1(B_0,n_1+i) \\
  &\le \sum_{j=0}^{k_1(B_0)-1} \frac{2b}{a} \frac{\alpha}{\eta} \left(\frac{(1-c_1) r(B_0)}{3}\right)^\delta g(u_{n_1+j}) \, \#C_1(B_0,n_1+i) \\
  &< \frac{1}{4} \, \#C_1(B_0,n_1+i).
\end{align*}

By \eqref{as1}, \eqref{num1} and the \( u \)-regularity of \( \rho \), the contribution of the second summation is:
\begin{align*}
  &\le \sum_{j \in J_2} \frac{2b3^\delta}{a} \left(\frac{\rho(u_{n_1+i})}{\rho(u_{n_1+j})}\right)^{\delta} \#C_1(B_0,n_1+i) \\
  &\le \frac{2b3^\delta}{a} \, \#C_1(B_0,n_1+i) \sum_{j=0}^{i-1} \lambda_2^{\delta(i-j)} \\
  &\le \frac{2b3^\delta}{a} \, \#C_1(B_0,n_1+i) \sum_{l=1}^{\infty} (\lambda_2^{\delta})^l \\
  &< \frac{1}{4} \, \#C_1(B_0,n_1+i).
\end{align*}
Combining the two cases yields \( \#U_1(B_0, n_1+i) < \frac{1}{2} \, \# C_1(B_0, n_1+i) \), which completes the proof of the claim.

Then, setting
$$\K_1(B_0, n_1+i):=\left\{ B\left(x, \frac{1-c_1}{3} \phi(R(\xi_x))\right): x\in G_1(B_0,n_1+i) \right\},$$  the  sub-level $K_1(B_0, n_1+i)$ is  defined to be
$$ K_1(B_0, n_1+i):=\bigcup_{B\in \K_1(B_0, n_1+i) }B .$$  Note that by construction for $0\le j\neq i \le k_1(B_0)$
$$K_1(B_0, n_1+j)\cap K_1(B_0, n_1+i)=\emptyset.$$

Finally, we define
\[
\mathcal{K}_1 = \mathcal{K}_1(B_0) = \bigcup_{i=0}^{k_1(B_0)} \mathcal{K}_1(B_0, n_1+i),
\]
and define the first level \( K_1 \) of the Cantor set to be
\[
K_1 :=  \bigcup_{B \in \mathcal{K}_1}  B.
\]

\subsubsection{Construction of   the second level $K_2$}

The second level is constructed by localizing the construction of the first level. More precisely, within each ball   $B_1$ in $\K_1,$  we construct a collection $\K_2(B_1)$ of smaller balls  contained in $B_1$ in a manner analogous to the construction of the first level, and we define
\begin{align}\label{newset1}\K_2:=\bigcup_{B_1 \in \K_1} \K_2(B_1) \quad \text{and} \quad K_2:=\bigcup_{B\in \K_2} B.\end{align}

For \( B_1 \in \mathcal{K}_1 \), denote by \( m_1(B_1) \) the constant associated with \( B_1 \) as in Lemma \ref{lem2}. We initiate the detailed construction by selecting \( n_2 > n_1 \) sufficiently large so that for any \( B_1 \in \mathcal{K}_1 \),
\(
n_2 > m_1(B_1),
\)
and
\begin{equation*}\label{2cc3}
  \frac{f(\phi(u_{n_2}))}{\phi(u_{n_2})^\delta} > \frac{3^{\delta}}{\alpha} \frac{f(r(B_1))}{r(B_1)^\delta}.
\end{equation*}

Since for \( B_1 \in \mathcal{K}_1 \), \( r(B_1) \le \phi(u_{n_1}) \), combining this with \eqref{cc1} and \eqref{cc3}, the definition of \( G^* \), and the monotonicity of \( f(r)/r^\delta \), we conclude that
\begin{equation*}\label{2cc1}
  g(u_{n_2}) \le G^* < \frac{a}{16b} \frac{1}{\alpha} \frac{f(r(B_1))}{r(B_1)^\delta} \left( \frac{3}{(1-c_2)} \right)^\delta.
\end{equation*}

For a fixed ball \( B_1 \in \mathcal{K}_1 \) and for any \( n \ge n_2 \), we define \( C_2(B_1, n) \) analogously to \( C_1(B_0, n) \). The main adjustments are that the constant \( c_2 \) now replaces \( c_1 \), and Lemma  \ref{lem2} is applied in place of Lemma  \ref{lem1}. We provide the details as follows.

Let \( k_2(B_1) \ge 1 \) be the unique integer such that
\begin{equation*}\label{2cc4}
\frac{2b \alpha}{a} \frac{r(B_1)^\delta}{f(r(B_1))} \left( \frac{(1-c_2)}{3} \right)^\delta \sum_{i=0}^{k_2(B_1)-1} g(u_{n_2+i}) < \frac{1}{4}
\end{equation*}
and
\begin{equation*}
\frac{2b \alpha}{a} \frac{r(B_1)^\delta}{f(r(B_1))} \left( \frac{(1-c_2)}{3} \right)^\delta \sum_{i=0}^{k_2(B_1)} g(u_{n_2+i}) \ge \frac{1}{4}.
\end{equation*}

For each \( n \ge n_2 \; (\ge m_{1}(B_1)) \), we find a subset \( Q_{B_1}(n) \subset J_u(n) \) fulfilling Conditions (1)--(4) in Lemma \ref{lem2}.
For each \( \xi \in Q_{B_1}(n) \), there exists \( x \in A_2(\xi) \) satisfying
\begin{equation}\label{s2}
B\left(x, \frac{1-c_2}{3} \phi(R(\xi))\right) \subset A_2(\xi) \subset B(\xi, \phi(u_{n})).
\end{equation}
Denote by \( C'_2(B_1, n) \) the collection of such centers \( x \) as \( \xi \) runs through \( Q_{B_1}(n) \).
Thus,
\[
\# C'_2(B_1, n) = \# Q_{B_1}(n) \ge a_1 \frac{r(B_1)^{\delta}}{\rho(u_{n})^{\delta}}.
\]
 We then take $C_2(B_1,n)$ as any sub-collection of $C'_2(B_1,n)$ satisfying 
\begin{equation*}\label{2num1}
 \frac{a_1}{2} \left(\frac{r(B_1)}{\rho(u_{n})}\right)^{\delta} \le \#C_2(B_1,n)\le a_1 \left(\frac{r(B_1)}{\rho(u_{n})}\right)^{\delta}.
\end{equation*}
For  $x \in C_2(B_1, n),$   there exists a unique $\xi_x \in Q_{B_1}(n)$ such that
$$B\left(x,  \phi(R(\xi_x))\right) \subset B(x,\phi(u_n)) \subset B(x, \rho(u_n))\subset B(\xi_x, 2\rho(u_n)). $$ Hence, the balls $B(x, 3\rho(u_{n}))$ with $x \in C_2(B_1, n)$ are pairwise disjoint.

\smallskip
\textbullet \ {\em{The local sub-level $\K_2(B_1, n_2).$}}
We define  $$G_2(B_1, n_2):=C_2(B_1, n_2).$$

The  sub-level $\K_2(B_1, n_2)$ is defined to be
$$ K_2(B_1, n_2):=\bigcup_{B \in  \K_2(B_1, n_2)} B ,$$
where
$$ \K_2(B_1, n_2):=\left\{ B\left(x, \frac{1-c_2}{3} \phi(R(\xi_x))\right):x\in G_2(B_1,n_2) \right\}.$$  

\smallskip
\textbullet \ {\em{The local sub-level $\K_2(B_1, n_2+i)$ with $1\le i \le k_2(B_1)$.}}

Assuming that, for $0\le j \le i-1$, the local sub-levels
$$\K_2(B_1,n_2+j):=\left\{ B\left(x, \frac{1-c_2}{3}\phi(R(\xi_x)) \right):x \in G_2(B_1, n_2+j)  \right \}$$
are already constructed, we proceed to construct $\K_2(B_1,n_2+i).$

Putting
$$h_{B_1}(n_2+i-1):=\left( \frac{\alpha } {f(r(B_1))} \left(\frac{(1-c_2)r(B_1)}{3}\right)^\delta f(\phi(u_{n_2+i-1})) \right)^{\frac{1}{\delta}},$$
we have
$$B\left(x,\frac{1-c_2}{3}\phi(R(\xi_x))\right)
\subset B(x,h_{B_1}(n_2+i-1)) \subset B(x,\rho(u_{n_2+i-1})).$$

We define
\begin{equation*}\label{set1}T_2(B_1, n_2+i-1):=\{B(x,h_{B_1}(n_2+i-1)): x \in G_2(B_1, n_2+i-1) \},\end{equation*}
and define $U_2(B_1, n_2+i)$ to be the set of \( y \in C_2(B_1, n_2+i) \) such that
\( B(y, 3\rho(u_{n_2+i})) \) intersects some thickening \( B(x, h_{B_1}(n_2+j)) \) in
\( T_2(B_1, n_2+j) \) with \( j = 0, 1, \ldots, i-1 \).
Set
\begin{align*}\label{set3} G_2(B_1, n_2+i):=C_2(B_1, n_2+i)\setminus U_2(B_1, n_2+i).\end{align*}
Following the same argument as in the first level (with obvious modifications to the estimates), it is routine to check that   $$\#G_2(B_1, n_2+i)\ge \frac{1}{2}\# C_2(B_1, n_2+i).$$

Set
\[
\mathcal{K}_2(B_1, n_2+i) := \left\{ B\left(x, \frac{1-c_2}{3} \phi(R(\xi_x)) \right) : x \in G_2(B_1, n_2+i) \right\}.
\]
Then, the local sub-level \( K_2(B_1, n_2+i) \) is defined to be
\[
K_2(B_1, n_2+i) := \bigcup_{B \in \mathcal{K}_2(B_1, n_2+i)} B.
\]
By construction, for \( 0 \le j \neq i \le k_2(B_1) \),
\[
K_2(B_1, n_2+j) \cap K_2(B_1, n_2+i) = \emptyset.
\]
Moreover, by virtue of Condition (4) in Lemma \ref{lem2} and the first inclusion in \eqref{s2}, when \( B_1 \in \mathcal{K}_1(B_0, n_1+i_1) \) and \( x \in G_2(B_1, n_2+i_2)\), for
\( z \in B\left(x, \frac{1-c_2}{3} \phi(R(\xi_x))\right) \) and \(\eta\in J_u(n_1+i_1,n_2+i_2)\), we have
\[
d(z,\eta) \ge c_1 \phi(R(\eta)).
\]

Finally, we define
\[
\mathcal{K}_2(B_1) := \bigcup_{i=0}^{k_2(B_1)} \mathcal{K}_2(B_1, n_2+i),
\]
and define the second level $K_2$ of the Cantor set as in (\ref{newset1}).

\subsubsection{Construction of a general   level $K_l$}\label{sec2}

Assume that $\K_1, \cdots ,\K_{l-1}$ have been constructed, each consisting of collections of sub-balls of the balls in the previous level.
For each \( B_{l-1}\in\K_{l-1}\), by the same localization procedure as in the second level (an outline of the construction will be given below), we define the family of sets  \begin{align*} \K_l(B_{l-1}):=\bigcup_{i=0}^{k_l(B_{l-1})}\K_l(B_{l-1},n_l+i).\end{align*}
The  $l$-level  of the Cantor set is then given by
$$\K_l:=\bigcup_{B_{l-1}\in \K_{l-1}} \K_l(B_{l-1})  \quad  \text{and} \quad K_l=\bigcup_{B\in \K_l}B.$$


We begin the construction by letting  \( n_l > n_{l-1} \) be a sufficiently large integer such that for any $B \in \mathcal{K}_{l-1}$,
\(n_l >   m_1(B)\)
and
\begin{equation*}
\frac{f(\phi(u_{n_l}))}{\phi(u_{n_l})^\delta} > \frac{3^{\delta}}{\alpha} \frac{f(r(B))}{r(B)^\delta}.
\end{equation*}
We conclude for any $B \in \mathcal{K}_{l-1}$  that
\begin{equation}\label{f4}
  g(u_{n_l}) <G^*<\frac{a}{16b} \frac{1}{\alpha}  \frac{f(r(B ))}{r(B)^\delta} \left(\frac{3}{1-c_l}\right)^\delta.
\end{equation}

We now fix  \( B_{l-1}\in\K_{l-1}\) and  sketch the construction of  $\K_l(B_{l-1})$.
Denote by $k_l(B_{l-1})\ge 1$   the unique integer such that
\begin{equation}\label{f3}
\frac{2b \alpha }{a}  \frac{r(B_{l-1})^\delta }{f(r(B_{l-1}))} \left(\frac{1-c_l}{3}\right)^\delta \sum\limits_{i=0}^{k_l(B_{l-1})-1} g(u_{n_l+i})  < \frac{1}{4}
\end{equation}
and
\begin{equation}\label{newcc5}
\frac{2b \alpha }{a}  \frac{r(B_{l-1})^\delta }{f(r(B_{l-1}))} \left(\frac{1-c_l}{3}\right)^\delta \sum\limits_{i=0}^{k_l(B_{l-1})}  g(u_{n_l+i})\ge \frac{1}{4}.
\end{equation}

Following the same procedure as in the construction of the second level, for each \( n \ge n_l \) we invoke Lemma \ref{lem2} for $B_{l-1}$ with the constant $c_l$ to define the set \( C_l(B_{l-1},n) \).

Set  \( G_l(B_{l-1},n_l)=C_l(B_{l-1},n_l)\). And for $1\le i\le k_l(B_{l-1})$, we define  $$h_{B_{l-1}}(n_l+i-1):=\left( \frac{\alpha } {f(r(B_{l-1}))} \left(\frac{(1-c_l)r(B_{l-1})}{3}\right)^\delta f(\phi(u_{n_l+i-1})) \right)^{\frac{1}{\delta}}.$$
We then introduce the sets $T_l(B_{l-1}, n_l+i-1)$ and $U_l(B_{l-1}, n_l+i)$, which are used to define \(  G_l(B_{l-1}, n_l+i) \) recursively.

Based on the same reasoning as earlier, we have  that
\begin{equation}\label{num3}\# G_l(B_{l-1}, n_l+i)\ge \frac{1}{2}\# C_l(B_{l-1}, n_l+i)\ge \frac{a_1}{4} \left(\frac{r(B_{l-1})}{\rho(u_{n_l+i})}\right)^{\delta} ,\end{equation} and for each \(x\in G_l(B_{l-1}, n_l+i),\)
\[ B\left(x,\frac{(1-c_l)}{3}\phi(R(\xi_x) )\right) \subset B(x, h_{B_{l-1}}(n_l+i))\subset B(x, \rho(u_{n_l+i} )),\]  while the balls \(B(x, 3\rho(u_{n_l+i}))\) with \(x \in G_l(B_{l-1}, n_l+i)\) are  pairwise disjoint.

For  $B_{l-1}\in \K_{l-1},$ define
$$\K_l(B_{l-1}, n_l+i):=\left\{ B\left(x, \frac{1-c_l}{3} \phi(R(\xi_x))\right): x\in G_l(B_{l-1},n_l+i) \right\} $$ and
$$K_l(B_{l-1}, n_l+i)=\bigcup_{B \in \K_l(B_{l-1}, n_l+i) }B .$$
Likewise,  for $0\le j\neq i \le k_l(B_{l-1})$,
$$K_l(B_{l-1}, n_l+j)\cap K_l(B_{l-1}, n_l+i)=\emptyset.$$  Additionally, 
when $B_{l-1}\in \K_{l-1}(B_{l-2},n_{l-1}+i_{l-1})$ and $x\in G_l(B_{l-1},n_l+i_l)$, for   $z\in B\left(x, \frac{1-c_l}{3} \phi(R(\xi_x))\right) $ and $\eta\in J_u(n_{l-1}+i_{l-1}, n_l+i_l)$, we have $$d(z,\eta)\ge c_{l-1} \phi(R(\eta)).$$

\subsubsection{Construction of the Cantor set} It is easily seen that $$K_l \subset K_{l-1} .$$ The desired Cantor set is defined as
$$ K_{\eta} :=\bigcap_{l \ge 1} K_l .$$

\begin{pro}
  We have $K_{\eta} \subset E(Q,R,\phi).$
\end{pro}
\begin{proof}

Take \( z \in K_{\eta} \). For \( \epsilon > 0 \), choose \( l_0 \ge 2 \) such that \( 1 - \epsilon < c_{l-1} \) holds for all \( l \ge l_0 \). Since \( z \in \bigcap_{l \ge l_0} K_l \), for every \( l \ge l_0 \) there exist a ball \( B_{l-1} \in \mathcal{K}_l(B_{l-2}, n_{l-1}+i_{l-1}) \) (where \( B_{l-2} \in \mathcal{K}_{l-2} \) and \( 0 \le i_{l-1} \le k_{l-1}(B_{l-2}) \)) and an integer \( i_l \) with \( 0 \le i_l \le k_l(B_{l-1}) \) such that \( z \in K_l(B_{l-1}, n_l+i_l) \). This means that for some \( x_l \in G_l(B_{l-1}, n_l+i_l) \),
\begin{equation}\label{na}
z \in B\left(x_l, \frac{1-c_l}{3} \phi(R(\xi_{x_l}))\right) \subset A_l(\xi_{x_l}) \subset B(\xi_{x_l}, \phi(R(\xi_{x_l}))).
\end{equation}
Then, by Condition (4) in Lemma \ref{lem2}, for any \( \eta \in J_u(n_{l-1}+i_{l-1}, n_{l}+i_{l}) \),
\[
d(z,\eta) \ge c_{l-1} \phi(R(\eta)) > (1-\epsilon) \phi(R(\eta)).
\]
Since this holds for all \( l \ge l_0 \), we conclude that for any \( \eta \in J_u(n_{l_0-1}+i_{l_0-1}, \infty) \) (i.e., \( \eta \in Q \) with \( R(\eta) < u_{n_{l_0-1}+i_{l_0-1}} \)),
\[
d(z,\eta) > (1-\epsilon) \phi(R(\eta)).
\]

On the other hand, for \( x_l \in G_l(B_{l-1}, n_l + i_l) \), the corresponding point \( \xi_{x_l} \) lies in \( Q_{B_{l-1}}(n_l + i_l) \subset Q \).  From the construction we deduce that \( \xi_{x_l} \in A_{l-1}(\xi_{x_{l-1}}) \), and hence the points \( \xi_{x_l} \) are distinct for distinct \( l \).
Consequently, by \eqref{na}, we obtain
\[
d(z,\gamma) < \phi(R(\gamma)) \quad \text{ for infinitely many } \gamma \in Q.
\]
This concludes the proof. \end{proof}

Next, we  establish a simple yet key  geometric lemma.

\begin{lem}\label{sqe}
Let \( l \ge 1 \) and \( B_{l-1} \in \mathcal{K}_{l-1} \). Then for any distinct balls \( B \in \mathcal{K}_l(B_{l-1}, n_l + i) \) and \( B' \in \mathcal{K}_l(B_{l-1}, n_l + i') \) with \( i' \ge i \), we have
\[
d(B,B') := \inf_{x \in B, \ y \in B'} d(x,y) \ge 2\rho(u_{n_l + i'}).
\]
\end{lem}
\begin{proof} Write $B=B\left(x_0, \frac{1-c_l}{3}\phi(R(\xi_{x_0}))\right)$ and $B'=B\left(y_0, \frac{1-c_l}{3}\phi(R(\xi_{y_0}))\right).$ 

When $i=i'$, recalling  that  $B\left(x, \frac{1-c_l}{3}\phi(R(\xi_x))\right) \subset B(x, \rho(u_{n_l+i} ))$ and   the balls $B(x, 3\rho(u_{n_l+i}))$ with $x \in G_l(B_{l-1}, n_l+i)$ are pairwise disjoint, we conclude  that
  $$ d(B,B')\ge 4 \rho(u_{n_l+i}) .$$

When \( i' > i \), we have that
\(
B = B\left(x_0, \frac{1-c_l}{3}\phi(R(\xi_{x_0}))\right) \subset B(x_0, h_{B_{l-1}}(n_l+i)),
\)
\(
B' = B\left(y_0, \frac{1-c_l}{3}\phi(R(\xi_{y_0}))\right) \subset B(y_0, \rho(u_{n_l+i'})),
\)
and
\(
B(y_0, 3\rho(u_{n_l+i'})) \cap B(x_0, h_{B_{l-1}}(n_l+i)) = \emptyset.
\)
Combining these facts yields
\[
d(B,B') \ge 2\rho(u_{n_l+i'}).
\]
\end{proof}

\subsection{A measure on $K_{\eta}$}
 The objective of this section is to construct a probability measure $\nu$ with support contained in $K_{\eta}$
  satisfying \eqref{g2}, namely
 \begin{equation}\label{goali}\nu (A)\le \frac{6^\delta}{  \alpha } \frac{f(r(A))}{\eta} \end{equation}
  for any ball $A$ with  sufficiently small radius $r(A).$

We define the probability measure \( \nu \) recursively by assigning masses to the balls \( B \in \mathcal{K}_l \).

For $l=0$, $\K_0$ contains only the element $B_0$, and we set $\nu(B_0):=1.$

For $l\ge1$ and $B \in \K_l$, there exists a unique ball $B_{l-1}\in \K_{l-1}$  such that $B \in \K_l(B_{l-1})$;   we   define the mass of  $B$ by
\begin{equation}\label{e1}
\nu(B) := \frac{f\left(\frac{3}{1-c_l}r(B)\right)}{ \sum_{B' \in \K_l(B_{l-1})} f\left(\frac{3}{1-c_l}r(B')\right)} \cdot \nu(B_{l-1}).
\end{equation}

By Kolmogorov's consistency theorem, the measure \( \nu \) admits a unique extension to a probability measure on \( K_{\eta} \). Explicitly, for any set \( E \subseteq X \),
\[
\nu(E)=\nu(E\cap K_{\eta}) :=\inf \left\{\sum_{i\ge1} \nu(B_i): E\subset \bigcup_{i\ge1}B_i, B_i \in \bigcup_{l\ge0}\K_l \right\}.
\]
\subsubsection{ The measure of a ball in $\K_l$}

We proceed to show that \eqref{goali} holds for any ball $B \in \K_l.$

\begin{lem}\label{measure1} We have the following inequalities:
\begin{equation*}\label{meq1}  \sum_{B \in \K_1(B_0)} f\left(\frac{3}{1-c_1}r(B)\right)\ge \left( \frac{3}{1-c_1}\right)^\delta \eta, \end{equation*}
  and for  $l\ge 2,$
  \begin{equation*} \sum_{B \in \K_l(B_{l-1})} f\left(\frac{3}{1-c_l}r(B)\right) \ge f\left( r(B_{l-1}) \right) \left(\frac{3}{1-c_{l}}\right)^{\delta}. \end{equation*}
\end{lem}
\begin{proof}
Recall that for  $l\ge1$ and $B\in \K_l(B_{l-1},n_l+i),$
\begin{equation}\label{lr}
    \frac{1-c_l}{3} \phi(u_{n_l+i+1}) \le r(B) \le \frac{1-c_l}{3} \phi(u_{n_l+i}),
\end{equation}
and
$$\#G_l(B_0, n_l+i)\ge \frac{1}{2}\#C_l(B_0, n_l+i) .$$

Using \eqref{num1} together with the \( u \)-regularity of \( \rho \), we obtain
  \begin{align}\label{nf1}
    \sum_{B \in \K_1(B_0)} f\left(\frac{3}{1-c_1}r(B)\right) &\ge \sum_{i=0}^{k_1(B_0)}\sum_{x \in G_1(B_0, n_1+i )}  f(\phi(u_{n_1+i+1})) \\\nonumber
     &\ge \frac{a_1}{4}  r(B_0)^\delta \sum_{i=0}^{k_1(B_0)} \frac{f(\phi(u_{n_1+i+1}))}{\rho(u_{n_1+i})^\delta}\\ \nonumber
    &\ge \frac{a_1 \lambda_1^\delta}{4}  {r(B_0)}^\delta \sum_{i=1}^{k_1(B_0)} g(u_{n_1+i}).
  \end{align}
  Combining \eqref{cc1} and \eqref{cc5} leads to
  \begin{equation}\label{nf2}
    \sum_{i=1}^{k_1(B_0)} g(u_{n_1+i}) \ge \frac{a}{16b} \frac{\eta}{\alpha} \left(\frac{3}{(1-c_1)r(B_0)}\right)^\delta .
  \end{equation}
  Substituting \eqref{defalpha} and \eqref{nf2} into \eqref{nf1} gives the first inequality of the lemma.

  For $l\ge 2,$ applying \eqref{num3} and the \( u \)-regularity of \( \rho \) yields
  \begin{align*}
    \sum_{B \in \K_l(B_{l-1})} f\left(\frac{3}{1-c_l}r(B)\right)  &\ge \sum_{i=0}^{k_l(B_{l-1})} \sum_{x \in G_l(B_{l-1}, n_l+i)}  f(\phi(u_{n_l+i+1}))  \\
    &\ge \frac{a_1 \lambda_1^\delta}{4}   r(B_{l-1})^\delta \sum_{i=1}^{k_l(B_{l-1})}g(u_{n_l+i}).
  \end{align*}
  Then, by \eqref{defalpha}, \eqref{f4} and \eqref{newcc5}, we deduce
\begin{align*}
    \sum_{B \in \K_l(B_{l-1})} f\left(\frac{3}{1-c_l}r(B)\right)  &\ge \frac{a_1 a \lambda_1^\delta }{64 b\alpha } \left(\frac{3}{1-c_l}\right)^\delta f(r(B_{l-1}))\\
    &\ge \left(\frac{3}{1-c_l}\right)^\delta f(r(B_{l-1})),
\end{align*}
which completes the proof of the second inequality.
\end{proof}

It follows from \eqref{e1}, Lemma \ref{measure1} and the monotonicity of \( f(r)/r^\delta \) that for any ball \( B_1 \in \mathcal{K}_1 \),
\begin{equation}\label{m1}
\frac{\nu(B_1)}{f(r(B_1))} \le
\frac{\nu(B_1)}{\left(\frac{1-c_1}{3} \right)^\delta f\left(\frac{3}{1-c_1} r(B_1)\right)} \le \frac{1}{\eta}.
\end{equation}

For \( l \ge 2 \), we obtain that for any ball \( B \in \mathcal{K}_l(B_{l-1}) \),
\begin{equation}\label{m2}
\frac{\nu(B)}{f(r(B))} \le \frac{\nu(B)}{\left(\frac{1-c_l}{3}\right)^\delta f\left(\frac{3}{1-c_l} r(B)\right)} \le \frac{\nu(B_{l-1})}{f(r(B_{l-1}))}.
\end{equation}

Using induction based on \eqref{m1} and \eqref{m2}, we conclude that for any \( B \in \mathcal{K}_l \),
\begin{equation}\label{ele}
\nu(B) \le \frac{\left(\frac{1-c_l}{3}\right)^\delta f\left(\frac{3}{1-c_l} r(B)\right)}{\eta} \le \frac{ f(r(B))}{\eta}.
\end{equation}

\subsubsection{Measure of a general ball}\label{sec3}
We now proceed to verify that inequality \eqref{goali} holds for any ball \( A \) centered in \( K_{\eta} \) with \( r(A) \le r_0 \).

 Since \( \nu \) is supported on \( K_{\eta} \), we may assume that \( A \cap K_{\eta} \neq \emptyset \); otherwise  \( \nu(A) = 0 \) and \eqref{goali} holds trivially.  Moreover, if \( A \) intersects only one ball in \( \mathcal{K}_l \) for every \( l \ge 1 \), then \( \nu(A) = 0 \). Consequently, we assume that there exists a unique integer \( l \ge 2 \) such that \( A \) intersects exactly one ball, say \( B_{l-1}\), in \( \mathcal{K}_{l-1} \) and at least two balls in \( \mathcal{K}_l \). The possibility that \( A \) intersects two or more balls in \( \mathcal{K}_1 \) can be ruled out by taking \( r(A) \) sufficiently small, as the balls in \( \mathcal{K}_1 \) are pairwise disjoint.

We may assume that
\begin{equation*}\label{feq1}r(A)<r(B_{l-1});\end{equation*}
otherwise, it follows immediately that
$$ \nu(A)\le \nu(B_{l-1}) \le \frac{f(r(B_{l-1}))}{\eta} \le \frac{f(r(A))}{\eta}.$$

Since \( A \) intersects only the ball \( B_{l-1} \) in \( \mathcal{K}_{l-1} \), any ball from \( \mathcal{K}_l \) that meets \( A \) must belong to the local level
\[
\mathcal{K}_l(B_{l-1}) = \bigcup_{i=0}^{k_l(B_{l-1})} \mathcal{K}_l(B_{l-1}, n_l + i).
\]
For each \( 0 \le i \le k_l(B_{l-1}) \), define
\[
N_i := \#\Big\{ B \in \mathcal{K}_l(B_{l-1}, n_l + i) : B \cap A \neq \emptyset \Big\}.
\]
Let \( i^* := \min\{ i : N_i > 0 \} \). Then, by \eqref{lr} and \eqref{ele}, we obtain
\begin{align}\label{f1}
  \nu(A) \le \sum_{i=i^*}^{k_l(B_{l-1})} N_i \left( \frac{1-c_l}{3} \right)^\delta \frac{ f(\phi(u_{n_l + i})) }{ \eta }.
\end{align}

Choose any ball \( B^* \in \mathcal{K}_l(B_{l-1}, n_l + i^*) \) that intersects \( A \).
If \( A \) also meets another distinct ball \( B\bigl(y, \frac{1-c_l}{3} \phi(R(\xi_y))\bigr) \in \mathcal{K}_l(B_{l-1}, n_l + i) \) for some \( i \ge i^* \), then Lemma \ref{sqe} yields
\[
r(A) \ge \rho(u_{n_l + i}),
\]
which in turn implies
\[
B\bigl(y, \rho(u_{n_l + i})\bigr) \subset 3A.
\]

Noting that the balls \( B\bigl(y, 3\rho(u_{n_l + i})\bigr) \) with \( y \in G_l(B_{l-1}, n_l + i) \) are pairwise disjoint, the Ahlfors regularity of \( \mu \) implies
\begin{equation}\label{Ni}
N_i \le \frac{b}{a} \left( \frac{3r(A)}{ \rho(u_{n_l + i})} \right)^{\delta}.
\end{equation}

We now consider two separate cases.

\emph{Case 1: }  $A$ intersects at least two  balls in $\K_l(B_{l-1},n_l+i^*).$

In this case, \eqref{Ni} holds for all \( i \ge i^* \), and it follows from \eqref{f1} that
\begin{align*}
  \nu(A) &\le \frac{b}{a} \frac{(3r(A))^\delta}{\eta} \left( \frac{1-c_l}{3} \right)^\delta \sum_{i=i^*}^{k_l(B_{l-1})} g(u_{n_l + i}) \\
  &\le \frac{b}{a} \frac{(3r(A))^\delta}{\eta} \left( \frac{1-c_l}{3} \right)^\delta \sum_{i=0}^{k_l(B_{l-1})} g(u_{n_l + i}).
\end{align*}
Using \eqref{f4}, \eqref{f3} and the definition of \( G^* \), we obtain
\[
\left( \frac{1-c_l}{3} \right)^\delta \sum_{i=0}^{k_l(B_{l-1})} g(u_{n_l + i}) \le \frac{a}{4b} \frac{1}{\alpha} \frac{f(r(B_{l-1}))}{r(B_{l-1})^\delta}.
\]
Consequently,
\begin{equation*}\label{fe2}
\nu(A) \le \frac{3^\delta r(A)^\delta}{4\alpha \eta} \frac{f(r(B_{l-1}))}{r(B_{l-1})^\delta} \le \frac{3^\delta}{4\alpha} \frac{f(r(A))}{\eta}.
\end{equation*}

\emph{Case 2.} $A$ intersects only one ball $B^*$ in $\K_l(B_{l-1},n_l+i^*).$

In this case, $N_{i^*}=1$. The terms in the sum on the right-hand side of \eqref{f1} with \( i > i^* \) can be estimated using the same argument as in Case 1, yielding
\begin{equation}\label{p1}
\sum_{i=i^*+1}^{k_l(B_{l-1})} \left( \frac{1-c_l}{3} \right)^\delta \frac{ f(\phi(u_{n_l + i})) }{ \eta } \le \frac{3^\delta}{4\alpha} \frac{ f(r(A)) }{ \eta }.
\end{equation}
However, the term $\left(\frac{1-c_l}{3}\right)^\delta {f(\phi(u_{n_l+i^*}))}/{\eta }$ corresponding to \( i = i^* \) requires a separate estimate, which we now provide.

Let \( B^* = B\left(x, \frac{1-c_l}{3} \phi(R(\xi_x))\right) \) be the unique ball in \( \mathcal{K}_l(B_{l-1}, n_l + i^*) \) intersecting \( A \).
Pick another ball \( B \in \mathcal{K}_l(B_{l-1}, n_l + i) \) with \( i > i^* \) that also meets \( A \). By construction,
\[
B^* \subset B\left(x, \frac{1-c_l}{3} \phi(u_{n_l + i^*})\right) \subset B\left(x, (1-c_l)\phi(u_{n_l + i^*})\right) \subset B(x, h_{B_{l-1}}(u_{n_l + i^*}))
\]
and
\[
B \cap B(x, h_{B_{l-1}}(u_{n_l + i^*})) = \emptyset.
\]
Then, by a simple geometric observation, we can assert that
\begin{equation}\label{f2}
r(A) \ge \frac{1-c_l}{3} \phi(u_{n_l + i^*}).
\end{equation}
Hence, combining \eqref{f2} and the monotonicity of \( f(r) \) and \( f(r)/r^\delta \), we deduce that
\begin{equation*}
 \left( \frac{1-c_l}{3} \right)^\delta \frac{f(\phi(u_{n_l + i^*}))}{\eta}
 \le \frac{1}{\eta} f\left( \frac{1-c_l}{3} \phi(u_{n_l + i^*}) \right)  \le \frac{f(r(A))}{\eta}.
\end{equation*}
Together with \eqref{p1}, this implies that \eqref{goali} holds for the ball \( A \), as desired.

\section{Proof of Theorem \ref{mainthm}: $G=\infty$}


The proof of Theorem \ref{mainthm} in the case \( G = \infty \) follows the same strategy as in the finite case. This is achieved by first fixing a ball \( B_0 \subset X \) of sufficiently small radius; then constructing a Cantor subset of \( B_0 \cap E(Q,R,\phi) \) that supports a certain probability measure; and finally applying the Mass Distribution Principle to obtain the desired result.

We remark that, unlike the finite case, the Cantor construction in this situation is considerably simpler, as it does not require the creation of sub-levels at each stage.

Henceforth, let \( B_0 \) be a fixed ball satisfying \( r(B_0) \le r_0 \). Then the local ubiquity condition holds for \( B_0 \).

\subsection{The Cantor Set $K_{\eta}$ }In this section, we construct a Cantor subset \( K_{\eta} \) of \( B_0 \cap E(Q,R,\phi) \) for a fixed \( \eta \ge 1 \).

We initiate the construction by setting \( \mathcal{K}_0 = \{ B_0 \} \) and \( K_0 = B_0 \).

The first level of the Cantor set is then constructed as follows.
Let \( \bar{n}_{B_0} \) and \( m_0 \) be the constants from Lemma \ref{lem1} and Lemma \ref{lem2}, respectively.
Recall that
\begin{equation}\label{neq}
    \limsup_{n \to \infty} g(u_n) =: G = \infty.
\end{equation}
Consequently, we can choose a sufficiently large integer \( n_1 > \max\{ \bar{n}_{B_0}, m_0 \} \) such that
\begin{equation}\label{neq1}
    \eta \, \varpi_1(B_0) \left( \frac{3}{\lambda_1 (1-c_1)} \right)^\delta \le g(u_{n_1+1}),
\end{equation}
where
\[
\varpi_1(B_0) := \frac{2}{a_1 \, r(B_0)^\delta}.
\]
With this choice of \( n_1 \), we construct the sets \( G_1(B_0, n_1) \) as outlined in \S \ref{sec1}. Define the collection \( \mathcal{K}_1 \) by
\[
\mathcal{K}_1 := \left\{ B\left(x, \frac{1-c_1}{3} \phi(R(\xi_x)) \right) : x \in G_1(B_0, n_1) \right\}.
\]
The first level \( K_1 \) of the Cantor set \( K_{\eta} \) is then defined as
\[
K_1 := \bigcup_{B \in \mathcal{K}_1} B.
\]

For \( l \ge 2 \), assume that the integers \( \{ n_j \}_{1 \le j \le l-1} \) and the families \( \{ \mathcal{K}_j \}_{j \le l-1} \) have already been constructed. Given \( B_{l-1} \in \mathcal{K}_{l-1} \), there exists a unique nested sequence of balls
\[
B_0, B_1, \dots, B_{l-1}
\]
such that \( B_i \in \mathcal{K}_i \) and \( B_{i+1} \subset B_i \) for all \( 0 \le i \le l-2 \). We then define
\begin{equation*}
   \varpi_l(B_{l-1}) := \frac{1}{r(B_0)^\delta} \left( \frac{2}{a_1} \right)^l \prod_{i=1}^{l-1} \left( \frac{\rho(u_{n_i})}{r(B_i)} \right)^\delta.
\end{equation*}

Next, we choose a sufficiently large integer \( n_l > n_{l-1} \) satisfying
\[
n_l > \max\{ m_1(B) : B \in \mathcal{K}_{l-1} \}
\]
and
\begin{equation}\label{neq2}
\max\left\{ \eta \, \varpi_l(B) \left( \frac{3}{\lambda_1 (1-c_l)} \right)^\delta : B \in \mathcal{K}_{l-1} \right\} \le g(u_{n_l+1}).
\end{equation}
Such an integer \( n_l \) exists by virtue of \eqref{neq} and the fact that \( \mathcal{K}_{l-1} \) is finite.

For each \( B_{l-1} \in \mathcal{K}_{l-1} \), we construct the set \( G_l(B_{l-1}, n_l) \) as described in \S \ref{sec2}. Define
\[
\mathcal{K}_l(B_{l-1}) := \left\{ B\left(x, \frac{1-c_l}{3} \phi(R(\xi_x)) \right) : x \in G_l(B_{l-1}, n_l) \right\}
\]
and
\[
\mathcal{K}_l := \bigcup_{B \in \mathcal{K}_{l-1}} \mathcal{K}_l(B_{l-1}).
\]
The \( l \)-th level of the Cantor set is then given by
\[
K_l := \bigcup_{B \in \mathcal{K}_l} B.
\]

Finally, the Cantor set \( K_{\eta} \) is defined as
\[
K_{\eta} := \bigcap_{l=0}^{\infty} K_l.
\]
Clearly,
\[
K_{\eta} \subset B_0 \cap E(Q,R,\phi).
\]

From the construction of \( G_l(B_{l-1}, n_l) \), we deduce that for each \( x \in G_l(B_{l-1}, n_l) \), 
\begin{equation}\label{ffeq2}
B\left(x, \frac{1-c_l}{3} \phi(R(\xi_x))\right) \subset B(x, \rho(u_{n_l}))
\end{equation}
and that the balls \( B(x, 3\rho(u_{n_l})) \) for \( x \in G_l(B_{l-1}, n_l) \) are pairwise disjoint. Furthermore,
\begin{equation}\label{neq3}
\# G_l(B_{l-1}, n_l) \ge \frac{a_1}{2} \left( \frac{r(B_{l-1})}{\rho(u_{n_l})} \right)^\delta \ge \frac{a_1}{2} \left( \frac{1-c_l}{3} \cdot \frac{\phi(u_{n_l+1})}{\rho(u_{n_l})} \right)^\delta,
\end{equation}
and for any \( l \ge 1 \) and \( B \in \mathcal{K}_l(B_{l-1}) \),
\begin{equation}\label{lr2}
\frac{1-c_l}{3} \phi(u_{n_l+1}) \le r(B) \le \frac{1-c_l}{3} \phi(u_{n_l}).
\end{equation}


\subsection{The measure $\nu$}
In this section, we construct a probability measure \( \nu \) supported on \( K_{\eta} \) that satisfies
\begin{equation}\label{goali2}
\nu(A) \le \frac{2b \, 3^\delta}{a_1 a} \frac{r(A)^s}{\eta}
\end{equation}
for every ball \( A \) of sufficiently small radius \( r(A) \).

We define the measure \( \nu \) recursively as follows. For \( l = 0 \), set \( \nu(B_0) = 1 \). For \( l \ge 1 \) and \( B \in \mathcal{K}_l \), let \( B_{l-1} \in \mathcal{K}_{l-1} \) be the unique ball containing \( B \). Then we define
\[
\nu(B) = \frac{\nu(B_{l-1})}{\# G_l(B_{l-1}, n_l)}.
\]

We first show that \eqref{goali2} holds for a  ball $B \in \K_l.$ By \eqref{lr2}, the monotonicity of \( f(r)/r^\delta \) and the $u$-regular property of $\rho,$  we obtain
\begin{align}\label{lf1}
    \rho(u_{n_l})^\delta \le \frac{1}{\lambda_1^\delta} \rho(u_{n_l+1})^\delta & \le \left(\frac{3} {\lambda_1(1-c_l)}\right)^\delta \frac{1}{g(u_{n_l+1})} \left(\frac {1-c_l}{3}\right)^\delta f(\phi(u_{n_l+1}))\\ \nonumber
    &\le \left(\frac{3} {\lambda_1(1-c_l)}\right)^\delta \frac{1}{g(u_{n_l+1})} f(r(B)).
\end{align} Let $B_{l-1} \in \mathcal{K}_{l-1}$ be the unique ball containing $B.$ Then, we deduce that
\begin{align}\label{nff1}
  \nu(B) &\overset{ \eqref{neq3}}{\le} \varpi_l(B_{l-1})\rho(u_{n_l})^{\delta} \\  \nonumber
  &\overset{\eqref{lf1}}{\le} \varpi_l(B_{l-1})  \left(\frac{3} {\lambda_1(1-c_l)}\right)^\delta \frac{1}{g(u_{n_l+1})} f(r(B)) \\ \nonumber
&\overset{\eqref{neq1},~\eqref{neq2}}{\le}    \frac{f(r(B))}{\eta}. \nonumber
\end{align}




To complete the proof, we now verify \eqref{goali2} for an arbitrary ball \( A \) centered in \( K_{\eta} \) with radius \( r(A) \le r_0 \). Following the same reasoning as in \S \ref{sec3}, we may assume that \( A \cap K_{\eta} \neq \emptyset \) and that there exists an integer \( l \ge 2 \) such that \( A \) intersects exactly one ball \( B_{l-1} \in \mathcal{K}_{l-1} \) and at least two balls in \( \mathcal{K}_l \). Consequently, we may assume
\begin{equation}\label{ffeq1}
      r(A) < r(B_{l-1}).
\end{equation}
From \eqref{ffeq2} and the fact that the balls \( B(x, 3\rho(u_{n_l})) \) with \( x \in G_l(B_{l-1}, n_l) \) are pairwise disjoint, we deduce that \( \rho(u_{n_l}) \le r(A) \). Therefore, we may assume that \( A \) intersects at least two balls \( B(x, \rho(u_{n_l})) \) with \( x \in G_l(B_{l-1}, n_l) \). Let \( N \) denote the number of such balls intersected by \( A \). Since the balls \( B(x, \rho(u_{n_l})) \) are pairwise disjoint, a standard geometric argument gives
\[
N \le \frac{b }{a} \left( \frac{3r(A)}{\rho(u_{n_l})} \right)^\delta.
\]

Consequently, using \eqref{nff1}, \eqref{ffeq1} and the monotonicity of \( f(r)/r^\delta \), we obtain
\begin{align}\label{ffeq4}
  \nu(A) &\le \sum_{\substack{B \in \mathcal{K}_l \\ B \cap A \neq \emptyset}} \nu(B)
  \le \varpi_l(B_{l-1}) \, \frac{3^\delta\,b}{a} \, r(A)^\delta \nonumber \\
  &\le f(r(A)) \, \varpi_l(B_{l-1}) \, \frac{r(B_{l-1})^\delta}{f(r(B_{l-1}))} \, \frac{3^\delta\,b}{a}.
\end{align}

Let \( B_{l-2} \in \mathcal{K}_{l-2} \) be the unique ball containing \( B_{l-1} \). Applying \eqref{lf1}, we have
\begin{align}\label{ffeq5}
   \varpi_l(B_{l-1}) \frac{r(B_{l-1})^\delta}{f(r(B_{l-1}))}
   &= \varpi_{l-1}(B_{l-2}) \, \frac{2}{a_1} \, \frac{\rho(u_{n_{l-1}})^\delta}{f(r(B_{l-1}))} \nonumber \\
   &\le \varpi_{l-1}(B_{l-2}) \, \frac{2}{a_1} \left( \frac{3}{\lambda_1(1-c_{l-1})} \right)^\delta \frac{1}{g(u_{n_{l-1}+1})}.
\end{align}

Combining \eqref{neq1}, \eqref{neq2}, \eqref{ffeq4} and \eqref{ffeq5}, we deduce that
\[
\nu(A) \le \frac{2\cdot 3^\delta\,b}{a_1 a} \, \frac{ f(r(A)) }{ \eta }.
\]

Thus, we have shown that
\[
\nu(A) \le \frac{2\cdot 3^\delta\,b}{a_1 a} \, \frac{ f(r(A)) }{ \eta }
\]
holds for every sufficiently small ball \( A \).

Applying the Mass Distribution Principle now yields
\[
\mathcal{H}^f( E(Q,R,\phi) ) \ge \mathcal{H}^f( K_{\eta} ) \ge \eta \, \frac{a_1 a}{2\cdot 3^\delta\,b}.
\]

Since \( \eta \ge 1 \) is arbitrary, we conclude that \( \mathcal{H}^f( E(Q,R,\phi) ) = \infty \), thereby completing the proof of Theorem \ref{mainthm} for the case \( G = \infty \).

\section{Applications}\label{appl}
We now turn to applications of our framework.
In what follows, the convergence parts follow directly by considering the natural cover of the limsup set \( W(Q,R,\phi) \). We will also repeatedly use the following elementary fact: if \( f \) is a positive monotone function, \( \alpha \in \mathbb{R} \), and \( k > 1 \), then the two series
\[
\sum_{n=1}^\infty k^{n\alpha} f(k^n) \quad \text{and} \quad \sum_{q=1}^\infty q^{\alpha-1} f(q)
\]
converge or diverge together.

\subsection{Simultaneous Diophantine approximation}
Let $\psi: \R^{+}\rightarrow \R^{+}$ be a non-increasing  function. Define
$$W(d,\psi):=\left\{\x\in [0,1]^d :\left|\x-\frac{\p}{q}\right|_{\infty}<\psi(q)~\text{for i.m. } (\p,q)\in \Z^d\times \N\right\},$$ where $\left|\x-\frac{\p}{q}\right|_{\infty}=\max_{1\le i \le d}\left|x_i-\frac{p_i}{q}\right|.$

To place this concrete setting within our general framework, we let
\[
(X,d) := ([0,1]^d, |\cdot|_{\infty}),\quad
Q := \left\{ \frac{\mathbf{p}}{q} : (\mathbf{p},q) \in \mathbb{Z}^d \times \mathbb{N} \right\},\quad
R:  \frac{\mathbf{p}}{q} \mapsto \frac{1}{q^2}.
\]
Let \( \phi: \mathbb{R}^+ \to \mathbb{R}^+ \) be a non-decreasing function such that \( \phi(1/q^2) = \psi(q) \). Then,
\[
E(Q,R,\phi) = \mathrm{Exact}(d,\psi) := W(d,\psi) \setminus \bigcup_{0 < \epsilon < 1} W(d,(1-\epsilon)\psi).
\]
Let the measure \( \mu \) be the \( d \)-dimensional Lebesgue measure and set \( \delta = d \). 

\begin{pro}\label{prop2}There is   $t\in (0,1)$ such that
  the system $(Q, R)$ is well-distributed in $X$ with respect to $(\rho,u)$, where $u=\{t^n\}$ and $\rho: r\mapsto r^{\frac{1+d}{2d}}.$
\end{pro}
\begin{proof}
It is straightforward to verify that the system \( (Q, R) \) is well-separated; it remains to prove that it is a local \( \mu \)-ubiquitous system relative to \( (\rho, u) \).

Fix a point \( \mathbf{z} = (z_1, \ldots, z_d) \in [0,1]^d \) and consider the ball
\[
B := B(\mathbf{z}, r) = \prod_{i=1}^d B(z_i, r) 
\subset [0,1]^d.
\]
Choose \( t\in(0,1) \) such that for all sufficiently large \( n \),
\begin{equation}\label{r1}
  2^{3d+1} t^{1/2} < \frac{1}{4} \quad \text{and} \quad 2^{2d+1} 3^d \, t^{\frac{n+1}{2}} \log (t^{-\frac{n+1}{2}})   < \frac{1}{4} r^d.
\end{equation}

 By Minkowski's theorem, for any \( \mathbf{x} \in B \), there exist an integer \( q \) with \( 1 \le |q| \le t^{-\frac{n+1}{2}} \) and integers \( p_1, \ldots, p_d \in \mathbb{Z} \) such that
\[
\left| x_i - \frac{p_i}{q} \right| < \frac{1}{|q|} \, t^{\frac{n+1}{2d}}, \quad 1 \le i \le d.
\]
Note that this implies \( -|q| \le p_i \le |q| \) for each \( i \).

Define
\[
A :=\bigcup_{1 \le |q| \le t^{-\frac{n}{2}}} \;
\bigcup_{ -|q| \le p_1, \ldots, p_d \le |q| } \;
\prod_{i=1}^d B\!\left( \frac{p_i}{q}, \frac{1}{q} \, t^{\frac{n+1}{2d}} \right).
\]
Then we have the inclusion
\begin{align*}
  B \setminus A \subset &\;
  \bigcup_{q=\lfloor t^{-n/2} \rfloor+1}^{\lfloor t^{-(n+1)/2} \rfloor} \;
  \bigcup_{p_1=-q}^{q} \cdots \bigcup_{p_d=-q}^{q} \;
  \prod_{i=1}^d B\!\left( \frac{p_i}{q}, t^{\frac{n(d+1)}{2d}} \right) \\
  \subset &\;
  \bigcup_{\xi \in J_u(n)} B\!\left( \xi, \rho(u_n) \right).
\end{align*}

We claim that
\[
\mu\Biggl( \bigcup_{\xi \in J_u(n)} B\bigl(\xi, \rho(u_n)\bigr) \Biggr) \ge \frac{1}{2} \mu(B).
\]
To establish this, it suffices to show that
\(
\mu(A \cap B) < \frac{1}{2} \mu(B).
\)

Fix \( q \in \mathbb{N} \) and let \( N_q \) denote the number of balls of the form
\[
\prod_{i=1}^d B\!\left( \frac{p_i}{q}, \frac{1}{q} t^{\frac{n+1}{2d}} \right), \quad -q \le p_i \le q,
\]
that can intersect \( B \). Since the points \( \{ p_i/q : -q \le p_i \le q \} \) are \( 1/q \)-separated, we obtain
\begin{equation}\label{fq}
N_q \le (2rq + 3)^d \le 2^{2d} r^d q^d + 2^d 3^d.
\end{equation}

We estimate \( \mu(A \cap B) \) as follows.
\begin{align*}
\mu(A \cap B)
&\le \sum_{1 \le q \le t^{-n/2}} 2 N_q \, 2^d q^{-d} t^{\frac{n+1}{2}} \\[4pt]
&\stackrel{\eqref{fq}}{\le} \sum_{1 \le q \le t^{-n/2}} \Bigl( 2^{3d+1} r^d t^{\frac{n+1}{2}} + 2^{2d+1} 3^d q^{-d} t^{\frac{n+1}{2}} \Bigr) \\[4pt]
&\le 2^{3d+1} r^d t^{\frac{1}{2}} + 2^{2d+1} 3^d t^{\frac{n+1}{2}} \log\!\bigl( t^{-\frac{n+1}{2}} \bigr) \\[4pt]
&\stackrel{\eqref{r1}}{<} \frac{1}{2} r^d \le \frac{1}{2} \mu(B).
\end{align*}
This completes the proof.
\end{proof}

Combining the \( d \)-Ahlfors regularity of \( \mu \) and the \( u \)-regularity of \( \rho \) with Proposition \ref{prop2}, Theorem \ref{mainthm} yields the divergence parts of the following statements.

\begin{thm}Let \( f \) be a dimension function such that \( f(r)/r \) decreases monotonically to \( 0 \) as \( r \to 0 \).
  Let  $\psi: \R^{+}\rightarrow \R^{+}$ be a non-increasing  function such that $\sum_{q=1}^\infty q \psi(q)<\infty.$ Then,
  $$   \H^f(\mathrm{Exact}(1,\psi))=\left\{\begin{array}{ll}
    0 &\text{if}~\sum\limits_{q=1}^\infty qf(\psi(q)) < \infty; \\
    \infty &\text{if}~\sum\limits_{q=1}^\infty qf(\psi(q)) =\infty.
    \end{array}
  \right.$$
\end{thm}

\subsection{Diophantine approximation with restrictions}
An active area of research in Diophantine approximation concerns the restriction of the numerators and denominators of the rational approximants in the classical set \( W(1,\psi) \) to certain number-theoretic sets. Relevant results are systematically presented in Chapter 6 of \cite{Hm}.

We consider the following case.
Let \( \psi: \mathbb{N} \to \mathbb{R}^+ \) be a non-increasing function, and let \( \mathcal{A} = \{ b^n : n \ge 1 \} \) where \( b \ge 2 \) is an integer. Define
\[
W_{\mathbb{Z},\mathcal{A}}(1,\psi) := \left\{ x \in [0,1] : \left| x - \frac{p}{q} \right| < \psi(q) \text{ for infinitely many } (p,q) \in \mathbb{Z} \times \mathcal{A} \right\}.
\]

In terms of our general framework, we set
\[
(X,d) := ([0,1], |\cdot|), \quad
Q := \left\{ \frac{p}{q} : (p,q) \in \mathbb{N} \times \mathcal{A} \right\}, \quad
R: \frac{p}{q} \mapsto \frac{1}{q}.
\]
Let \( \phi: \mathbb{R}^+ \to \mathbb{R}^+ \) be a non-decreasing function satisfying \( \phi(1/q) = \psi(q) \). Then,
\[
E(Q,R,\phi) = \mathrm{Exact}_{\mathbb{Z},\mathcal{A}}(\psi) :=
W_{\mathbb{Z},\mathcal{A}}(1,\psi) \; \setminus \bigcup_{0 < \epsilon < 1} W_{\mathbb{Z},\mathcal{A}}(1,(1-\epsilon)\psi).
\]

Let the measure $\mu$ be one-dimensional Lebesgue measure and $\delta=1.$  

\begin{pro}
The system $(Q, R)$ is well-distributed in $X$ with respect to $(\rho,u)$, where $u=\{b^{-n}\}$ and $\rho: r\mapsto r.$
\end{pro}
\begin{proof}A straightforward verification shows that the system \( (Q, R) \) is well-separated. To establish its local ubiquity with respect to \( (\rho, u) \), we note that for any interval \( B \subset [0,1] \),
 $$\mu\left(B \cap \bigcup_{1\le i \le b^{n+1}} B\left(\frac{i}{b^{n+1}},\frac{1}{b^{n}} \right) \right) =\mu(B). $$\end{proof}


\begin{thm}\label{thm2}Let \( f \) be a dimension function such that \( f(r)/r \) decreases monotonically to \( 0 \) as \( r \to 0 \).
  Let  $\psi: \N \rightarrow \R^{+}$ be a non-increasing  function such that $\sum_{n=1}^\infty b^{n} \psi(b^n)<\infty.$ Then
  $$   \H^f(\mathrm{Exact}_{\Z,\A}(1,\psi))=\left\{\begin{array}{ll}
    0 &\text{if}~\sum\limits_{n=1}^\infty b^nf(\psi(b^n)) < \infty; \\
    \infty &\text{if}~\sum\limits_{n=1}^\infty b^nf(\psi(b^n)) =\infty.
    \end{array}
  \right.$$
\end{thm}
\begin{rem}
For \( \tau > 0 \), let \( \psi_\tau \) denote the function \( q \mapsto q^{-\tau} \). Rynne \cite{Ry98} proved that
\(
\dim_{\mathrm{H}} W_{\mathbb{Z},\mathcal{A}}(1,\psi_\tau) = \frac{1}{\tau}.
\)
Consequently, Theorem \ref{thm2} yields
\[
\dim_{\mathrm{H}} \mathrm{Exact}_{\mathbb{Z},\mathcal{A}}(1,\psi_\tau) = \dim_{\mathrm{H}} W_{\mathbb{Z},\mathcal{A}}(1,\psi_\tau).
\]
\end{rem}

\subsection*{Acknowledgements} This work is supported by National Key R$\&$D Program of China (No. 2024YFA1013700) and NSFC (No. 12171172).

\end{document}